\documentclass[a4paper,12pt]{amsart}

\title{The Euler system of cyclotomic units and higher Fitting ideals}
\author{Tatsuya Ohshita}
\address{Department of Mathematics, Kyoto University,
Kyoto 606-8502, Japan}
\email{ohshita@math.kyoto-u.ac.jp}
\date{\today}

\usepackage{amsmath}
\usepackage{amsmath}
\usepackage{amssymb}
\usepackage{amscd}
\usepackage{mathrsfs}
\usepackage[all]{xy}
\usepackage[dvipdfm]{graphicx}

\newtheorem{thm}{Theorem}[section]
\newtheorem{prop}[thm]{Proposition}
\newtheorem{cor}[thm]{Corollary}
\newtheorem{lem}[thm]{Lemma}

\newtheorem{claim}[thm]{Claim}
\theoremstyle{definition}
\newtheorem{dfn}[thm]{Definition}
\newtheorem{exa}[thm]{Example}
\newtheorem{rem}[thm]{Remark}

\def\plim{\mathop{\underleftarrow{\mathrm{lim}}}\nolimits}
\def\Gal{\mathop{\mathrm{Gal}}\nolimits}
\def\cha{\mathop{\mathrm{char}}\nolimits}
\def\Fitt{\mathop{\mathrm{Fitt}}\nolimits}
\def\det{\mathop{\mathrm{det}}\nolimits}
\def\Im{\mathop{\mathrm{Im}}\nolimits}

\def\Hom{\mathop{\mathrm{Hom}}\nolimits}

\def\ann{\mathop{\mathrm{ann}}\nolimits}
\def\Frob{\text{\rm Fr}}

\newcommand{\mf}[1]{{\mathfrak{#1}}}
\newcommand{\mb}[1]{{\mathbf{#1}}}
\newcommand{\bb}[1]{{\mathbb{#1}}}
\newcommand{\mca}[1]{{\mathcal{#1}}}
\newcommand{\mrm}[1]{{\mathrm{#1}}}

\begin{document}

\begin{abstract}
Kurihara established a refinement of the minus-part of the Iwasawa main conjecture for totally real number fields using the higher Fitting ideals (\cite{Ku}). 
In this paper, we study the higher Fitting ideals of the plus-part of the Iwasawa module associated to the cyclotomic $\bb{Z}_p$-extension of $\bb{Q}(\mu_p)$ for an odd prime number $p$ by similar methods as in \cite{Ku}. 
We define {\em the higher cyclotomic ideals} $\{ \mf{C}_i \}_{i \ge 0}$, which are ideals of the Iwasawa algebra defined by the Kolyvagin derivatives of cyclotomic units, and prove that they give upper bounds of the  higher Fitting ideals.
Our result can be regarded as a refinement of the plus-part of  the Iwasawa main conjecture for $\bb{Q}$.    
\end{abstract}

\maketitle

\section{Introduction}\label{notation}
The Iwasawa main conjecture describes the characteristic ideals of certain Iwasawa modules. 
The characteristic ideals give various important knowledge on the structure of finitely generated torsion Iwasawa modules, but these are not enough to determine the pseudo-isomorphism classes of them (cf.\ \S 2) completely. The pseudo-isomorphism classes of finitely generated torsion Iwasawa modules are determined by the Fitting ideals. In \cite{Ku}, Kurihara proved that all the higher Fitting ideals of the minus-part of the Iwasawa modules associated to the cyclotomic $\bb{Z}_p$-extension of certain CM-fields coincide with the higher Stickelberger ideals, which are defined by analytic objects arising from $p$-adic $L$-functions (cf.\ \cite{Ku} Theorem 1.1). His result is a refinement of the minus-part of the Iwasawa main conjecture for totally real number fields. 

In this paper, we study the plus-part of the Iwasawa modules by similar methods as in \cite{Ku}. We obtain upper bounds of the higher Fitting ideals of the Iwasawa modules. 
The main tool in \cite{Ku} is the Kolyvagin system of Gauss sums. Instead, in this paper, we use the Euler system of cyclotomic units, so we can only treat the Iwasawa modules associated to the cyclotomic $\bb{Z}_p$-extension of subfields of cyclotomic fields. 

We shall state the main theorem of this paper. Fix an odd prime number $p$. Let $\mu_n$ be the group of all $n$-th roots of unity contained in an algebraic closure $\overline{\bb{Q}}$ of $\bb{Q}$. For an integer $m$ with $m \ge 0$, let $F_m$ be the maximal totally real subfield of the cyclotomic field $\bb{Q}(\mu_{p^{m+1}})$. We denote the ring of integers of $F_m$ by $\mca{O}_{F_m}$, and the (unique) prime of $\mca{O}_{F_m}$ above $p$ by $\mf{p}_m$. We put $F_\infty := \bigcup_{m \ge 0}F_m$, and $\Gamma_m := \Gal(F_{\infty}/F_m)$ for any $m \ge 0$. Let $\Lambda:=\bb{Z}_p[[\Gal(F_\infty/\bb{Q})]]=\plim \bb{Z}_p[\Gal(F_m/\bb{Q})]$. We define a $\Lambda$-module $X:= \plim A_m$, where $A_m$ is the $p$-Sylow subgroup of the ideal class group of $F_m$ and the projective limit is taken with respect to the norm map. We say a $\Lambda$-module $M$ is {\em pseudo-null} if its order is finite.
Let $X_{\mrm{fin}}$ be the largest pseudo-null $\Lambda$-submodule of $X$, and $X':=X/X_{\mrm{fin}}$.
We study $X'$ instead of $X$ by a technical reason (cf.\ Lemma \ref{no pn}).

Put $\Delta := \Gal(F_0/\bb{Q})$. Since $\Gal(F_{\infty}/\bb{Q}) \simeq \Delta \times \Gamma_0$ and the order of $\Delta$ is prime to $p$, we have a decomposition $\Lambda=\bigoplus_{\chi}\Lambda_{\chi}$, where $\chi$ runs through all characters in $\widehat{\Delta} :=\Hom(\Delta, \bb{Z}_p^{\times})$, and $\Lambda_\chi$ is a $\bb{Z}_p$-algebra isomorphic to $\bb{Z}_p[[\Gamma_0]]\simeq\bb{Z}_p[[T]]$ on which $\Delta$ acts via $\chi$. Then for any $\Lambda$-module $M$, we decompose $M=\bigoplus_{\chi}M_{\chi}$, where $M_{\chi}:=M\otimes_{\Lambda}\Lambda_{\chi}$. 
Let $\ann_{\Lambda_{\chi}} (X_{\mrm{fin}, \chi})$ be the annihilator of $X_{\mrm{fin},\chi}$ as a $\Lambda_{\chi}$-module. 

In this paper, we study the higher Fitting ideals $\{\Fitt_{\Lambda_\chi,i}(X'_{\chi})\}_{i \ge 0}$ of $X'_{\chi}$ for a non-trivial character $\chi \in \widehat{\Delta}$. 
Note that $X_{\chi}$ and $X'_{\chi}$ belong to the same pseudo-isomorphism class, so the pseudo-isomorphism class of $X_{\chi}$ is determined by the family $\{\Fitt_{\Lambda_\chi,i}(X'_{\chi})\}_{i \ge 0}$ (cf.\ Example \ref{rei-Iwasawa}). 

In the case of the plus-part, a problem lies in how to define the ideals which are substitutes for the higher Stickelberger ideals because we do not have elements as the Stickelberger elements in group rings of Galois groups.  We shall define an ideal $\mf{C}_i$ of $\Lambda$, called {\em the higher cyclotomic ideals}  $\mf{C}_i$ for each $i \in \bb{Z}_{ \ge 0}$ in \S \ref{the section of cyclotomic ideals}, by using the Euler system of cyclotomic units (cf.\ Definition \ref{The $i$-th cyclotomic ideal}). Roughly speaking, first, we shall define the ideals $\mf{C}_{i,F_m,N}$  of the group ring $R_{F_m,N}:=\bb{Z}/p^N[\Gal(F_m/\bb{Q})]$ generated by images of certain Kolyvagin derivatives $\kappa_{m,N}(\xi)$ by {\em all $R_{F_m,N}$-homomorphisms} $\xymatrix{F_m^\times/(F_m^\times)^{p^N} \ar[r] & R_{F_m,N}}$, then we shall define $\mf{C}_i$ by the projective limit of them. 

The goal of this paper is to prove the following theorem. 
\begin{thm}
\label{Main theorem}
Let $\chi  \in \widehat \Delta$ be a non-trivial character. 
\begin{enumerate}
\item $\mf{C}_{0,\chi} \subseteq \Fitt_{\Lambda_\chi, 0}(X'_{\chi})$.
\item $\ann_{\Lambda_{\chi}} (X_{\mrm{fin}, \chi})\Fitt_{\Lambda_\chi, i}(X'_{\chi})\subseteq \mf{C}_{i, \chi}$ for  $i \ge 0$.
\end{enumerate} 
\end{thm}
For the case of the trivial character, see Remark \ref{structure of trivial character part}. 
By Theorem \ref{Main theorem}, for $i=0$, 
we obtain both upper and lower bounds of $\Fitt_{\Lambda_\chi, 0}(X'_{\chi})$. 
Since ``error terms" $\ann_{\Lambda_{\chi}} (X_{\mrm{fin}, \chi})$ is an ideal of $\Lambda_\chi$ whose index is finite, Theorem \ref{Main theorem} for $i=0$ determines the characteristic ideal of $X'_{\chi}$, which is equal to the characteristic ideal of $X_{\chi}$ (cf.\ Remark \ref{equivalence between Th and IMC}). 
Therefore, our theorem can be regarded as a refinement of the Iwasawa main conjecture. We use the Iwasawa main conjecture in the proof, so we do not give a new proof of the Iwasawa main conjecture. On the other hand, for $i \ge 1$, we obtain only upper bounds of $\Fitt_{\Lambda_\chi, i}(X'_{\chi})$.

Theorem \ref{Main theorem} gives some knowledge on the structure of the ``growing-part" $X'$ along the cyclotomic $\bb{Z}_p$-extension of $F_0$. But it gives nothing on the pseudo-null-part $X_{\rm{fin}}$.  
In particular, if the Greenberg conjecture holds (i.e.\ if $X_{\chi}$ is pseudo-null, for example, see \cite{Gr} Conjecture 3.4), then our theorem says nothing.

\begin{rem}
At the end of introduction, we remark on the case of $F_0$. 
Rubin determined the structure of $A_0$ (cf.\ \cite{Ru}, \cite{Ru2} and \cite{MR}). 
In our notation, Rubin's result implies that the ideals $\{\mf{C}_{i,F_0,N} \}_{i \ge 0}$ give the structure of $A_{0,\chi}$. 
For detail, see Remark \ref{0-layer}.  
\end{rem}

\subsection*{Notation}
In this paper, we use the following notation.

We fix an algebraic closure $\overline{\bb{Q}}$ of $\bb{Q}$. In this paper, an algebraic number field is a subfield of $\overline{\bb{Q}}$ which is a finite extension of $\bb{Q}$.

Let $L/K$ be a finite Galois extension of algebraic number fields. Let $\lambda$ be a prime ideal of $K$, and $\lambda'$ a prime ideal of $L$ above $\lambda$. We denote the completion of $K$ at $\lambda$ by $K_{\lambda}$. If $\lambda$ is unramified in $L/K$,  the geometric Frobenius at $\lambda'$ is denoted by $(\lambda', L/K) \in \Gal(L/K)$. (Note that some authors use the inverse of our $(\lambda', L/K)$.)

We fix a family of embeddings $\{ \xymatrix{\ell_{\overline{\bb{Q}}}\colon\overline{\bb{Q}}\ar@{^{(}->}[r] & \overline{\bb{Q}}_\ell} \}_{\ell: \rm{prime}}$ satisfying a technical condition (A) as follows.
\begin{itemize}
\item[(A)] For any subfield $K \subset \overline{\bb{Q}}$ which is a finite Galois extension of $\bb{Q}$ and any element $\sigma \in \Gal(K/\bb{Q})$, there exist infinitely many prime numbers $\ell$ such that $\ell$ is unramified in $K/\bb{Q}$ and $(\ell_K, K/\bb{Q})= \sigma$, where $\ell_K$ is the prime ideal corresponding to the  embedding $\ell_{\overline{\bb{Q}}}|_K$.
\end{itemize}
We can prove the existence of a family satisfying the condition (A) by the Chebotarev density theorem easily.

Let $\ell$ be a prime number. For an algebraic number field $K$, let $\ell_K$ be the prime ideal of $K$ corresponding to the  embedding $\ell_{\overline{\bb{Q}}}|_K$. Then, if $K_1 \supseteq K_2$ is an extension of algebraic number fields, we have $\ell_{K_1}| \ell_{K_2}$. 

For an abelian group $M$ and a positive integer $n$, we write $M/n$ in place of $M/nM$ for simplicity. In particular, for the multiplicative group $K^\times$ of a field $K$, we write $K^\times/p^N$ in place of $K^\times/(K^\times)^{p^N}$.

For a $\Lambda$-module $M$, we denote the $\Gamma_m$-invariants (resp.\ $\Gamma_m$-coinvariants) of $M$ by $M^{\Gamma_m}$ (resp.\  $M_{\Gamma_m}$). 

Let $R$ be a commutative ring.  
For an $R$-module $M$, we define $\ann_{R}(M)$ to be  annihilator of $M$. Namely, $$\ann_R(M):=\{ a \in R \ | \ am=0 \ \text{for any}\ m \in M \}.$$  

\section*{Acknowledgement}
The author would like to thank Professors Kazuya Kato, Masato Kurihara and Tetsushi Ito for their helpful advices.
This work was supported by Grant-in-Aid for JSPS Fellows (22-2753) from Japan Society for the Promotion of Science.

\section{Preliminaries}
In this section, we recall some preliminary results used in this paper. 
We use the same notation as in \S 1. In particular, $F_m$ is the maximal totally real subfield of the cyclotomic field $\bb{Q}(\mu_{p^{m+1}})$, and $\Lambda=\bb{Z}_p[[\Gal(F_\infty/\bb{Q})]]$.

\subsection{}
Here, we briefly recall the plus-part of the Iwasawa main conjecture.

First, we recall the definition of the characteristic ideals.
Let $\chi \in \widehat{\Delta}$ be a character.
Two $\Lambda_{\chi}$-modules $M$ and $N$ are said to be {\em pseudo-isomorphic} and we write $M \sim N$ if there is a $\Lambda_{\chi}$-homomorphism 
$\xymatrix{M \ar[r] & N}$ whose kernel and cokernel are pseudo-null. The relation $\sim $ is an equivalence relation of finitely generated torsion $\Lambda_{\chi}$-modules. For a finitely generated torsion $\Lambda_{\chi}$-module $M$, there exists a  finite sequence $f_1, f_2, \dots, f_n$  of non-zero elements of $\Lambda_{\chi}$ such that $f_i$ divides $f_{i+1}$ for $1 \le i \le n-1$ and  $M \sim \bigoplus_{i=1}^n\Lambda_{\chi}/f_i\Lambda_{\chi}$. We define the characteristic ideal $\cha_{\Lambda_{\chi}}(M)$ of $M$ by $$\cha_{\Lambda_{\chi}}(M):=f_1f_2\cdots f_n\Lambda_{\chi}.$$
Note that characteristic ideals are well-defined and depend on the pseudo-isomorphism classes.

To state the Iwasawa main conjecture, we next recall some preliminary results on unit groups. 

For a subgroup $M$ of the unit group $\mca{O}_{F_m}^\times$, we define $M^1$ by $$M^1:=\{a \in M \ |\  a \equiv 1 \pmod{\mf{p}_m} \}.$$ When $M=\mca{O}_{F_m}^\times$, we write $\mca{O}_{F_m}^1$ in place of $(\mca{O}_{F_m}^\times)^1$. 

We write $F_{\mf{p}_m}$ for the completion of $F_m$ at the place $\mf{p}_m$. For a subset $M$ of $F_m$, let $\overline{M}$ denote the closure of $M$ in $F_{\mf{p}_m}$. Let $C_{F_m}$ be the group of cyclotomic units of $F_m$ (cf.\ \cite{CS} Definition 4.3.1 or \cite{Wa} \S 8.1).  We define $$\overline{\mca{O}_{\infty}^1}:=\plim\overline{\mca{O}_{F_m}^1},$$ $$\overline{C_{\infty}^1}:=\plim\overline{C_{F_m}^1},$$ where the projective limit is taken with respect to the norm map. 

\begin{prop}
\label{modules}
\begin{enumerate}
\item The $\Lambda$-module  $X$ is finitely generated torsion.
\item The $\Lambda$-module $\overline{\mca{O}_{\infty}^1}$ is free of rank 1.
\item The $\Lambda$-module $\overline{C_{\infty}^1}$ is free of rank 1.
\end{enumerate}
\end{prop}

\begin{proof}
The first assertion is a special case of \cite{Wa} Lemma 13.18.
The second assertion is \cite{CS} Theorem 4.7.1.
For the last assertion, see the proof of \cite{CS} Theorem 4.4.1.
\end{proof}

Note that by Proposition \ref{modules} (2) and (3), the $\chi$-part $(\overline{\mca{O}_{\infty}^1}/\overline{C_{\infty}^1})_{\chi} = ( \overline{\mca{O}_{\infty}^1})_{\chi}/(\overline{C_{\infty}^1})_{\chi}$ is a finitely generated torsion $\Lambda_{\chi}$-module. Then, by the Proposition \ref{modules}, we can consider $\cha_{\Lambda_{\chi}}(X_{\chi})$ and $\cha_{\Lambda_{\chi}}\big((\overline{\mca{O}_{\infty}^1}/\overline{C_{\infty}^1})_{\chi}\big)$. The statement of the Iwasawa main conjecture is the following:
\begin{itemize}
\item[] 
$\cha_{\Lambda_{\chi}}(X_{\chi})=\cha_{\Lambda_\chi}\big((\overline{\mca{O}_{\infty}^1}/\overline{C_{\infty}^1})_{\chi}\big)$ \ for any character $\chi \in \widehat{\Delta}$.
\end{itemize}

\subsection{} 
Here, we recall some results on $\Gamma_m$-coinvariants (or invariants) of certain $\Lambda$-modules.
They play important roles in the technical aspects in this paper.
In particular, we need them when we determine ``error terms" in Theorem \ref{Main theorem}.
\begin{prop}[a special case of \cite{Wa} Proposition 13.22] 
\label{ideal class}
Let $m$ be an integer with $m \ge 0$. Then, we have the canonical isomorphism 
$$X_{\Gamma_m}\simeq A_m.$$
\end{prop}

For each $m \ge 0$, we define 
\begin{align*}
N_{\infty}(\mca{O}_{F_m}^\times) &:= \bigcap_{n\geq m}N_{F_n/F_m}(\mca{O}_{F_n}^\times), \\
N_{\infty}(\mca{O}_{F_m}^1) &:= \bigcap_{n\geq m} N_{F_n/F_m}(\mca{O}_{F_n}^1),  \intertext{and} 
N_{\infty}(\overline{\mca{O}_{F_m}^1}) &:= \bigcap_{n\geq m}N_{F_{\mf{p}_n}/F_{\mf{p}_m}}(\overline{\mca{O}_{F_n}^1}).
\end{align*}
Note that we have $N_{\infty}(\overline{\mca{O}_{F_m}^1}) =\overline{N_{\infty}(\mca{O}_{F_m}^1)}$.

\begin{prop}
\label{projection}
Let $m$ be an integer with $m \ge 0$.
\begin{enumerate}
\item The canonical homomorphism $\xymatrix{\mrm{pr}(m,C)\colon (\overline{C_{\infty}^1})_{\Gamma_m} \ar[r] & \overline{C_{F_m}^1}}$ is surjective. 
The kernel of $\mrm{pr}(m,C)$ is isomorphic to $\bb{Z}_p$ with the trivial action of $\Gal(F_\infty/\bb{Q})$. {\em (\cite{CS} Theorem 4.6.3.)}
\item The image of the canonical homomorphism $\xymatrix{\mrm{pr}(m,\mca{O}^1)\colon(\overline{\mca{O}_{\infty}^1})_{\Gamma_m} \ar[r] & \overline{\mca{O}_{F_m}^1}}$ is $N_{\infty}(\overline{\mca{O}_{F_m}^1})$.  The kernel of $\mrm{pr}(m,\mca{O}^1)$ is isomorphic to $\bb{Z}_p$ with the trivial action of $\Gal(F_\infty/\bb{Q})$. {\em (\cite{CS} Theorem 4.7.4.)}
\end{enumerate}
\end{prop}
The following corollary immediately follows from Proposition \ref{projection}.
\begin{cor}\label{coinv}
For all $m \ge 0$ and non-trivial character $\chi \in \widehat{\Delta}$,
we have the following canonical isomorphisms of $\bb{Z}_p[\Gal(F_m/\bb{Q})]$-modules:
\begin{enumerate}
\item $(\overline{C_{\infty}^1})_{\Gamma_m,\chi} \simeq  (\overline{C_{F_m}^1})_{\chi};$
\item $(\overline{\mca{O}_{\infty}^1})_{\Gamma_m,\chi} \simeq N_{\infty}(\overline{\mca{O}_{F_m}^1})_\chi$.
\end{enumerate}
\end{cor}

\begin{prop}[\cite{CS} Theorem 4.7.6]\label{delta no moto}
For all $m \ge 0$, we have the canonical isomorphism of $\bb{Z}_p[\Gal(F_m/\bb{Q})]$-modules
$$X^{\Gamma_m}\simeq \overline{\mca{O}_{F_m}^1}/N_{\infty}(\overline{\mca{O}_{F_m}^1}).$$
\end{prop}

\begin{cor}
\label{delta}
The $\Lambda$-module $\overline{\mca{O}_{F_m}^1}/N_{\infty}(\overline{\mca{O}_{F_m}^1})$ is annihilated by $\ann_{\Lambda}(X_{\mrm{fin}})$. 
\end{cor}

\begin{proof}[Proof of Corollary \ref{delta}]
Since $X_{\Gamma_m}$ is isomorphic to the finite group $A_m$ by Proposition \ref{ideal class}, the $\Lambda$-module $X^{\Gamma_m}$ is pseudo-null. Then,  $X^{\Gamma_m}$ is contained in $X_{\mrm{fin}}$. Hence Corollary \ref{delta} follows from Proposition \ref{delta no moto}.
\end{proof}

\begin{rem}
By Leopoldt's conjecture for $F_m$ (cf.\ \cite{Wa} Corollary 5.32), we have the natural isomorphism $$\xymatrix{\mca{O}_{F_m}^1\otimes\bb{Z}_p \ar[r]^(0.56){\simeq}  & \overline{\mca{O}_{F_m}^1}.
}$$
Then, we have the following isomorphisms:
\begin{enumerate}
\item $\xymatrix{ \mca{O}_{F_m}^\times\otimes\bb{Z}_p & \mca{O}_{F_m}^1\otimes\bb{Z}_p \ar[l]_(0.46){\simeq} \ar[r]^(0.56){\simeq}  & \overline{\mca{O}_{F_m}^1} ;
}$
\item $\xymatrix{ N_{\infty}(\mca{O}_{F_m}^\times)\otimes\bb{Z}_p &N_{\infty}( \mca{O}_{F_m}^1)\otimes\bb{Z}_p \ar[l]_(0.46){\simeq} \ar[r]^(0.55){\simeq}  & N_{\infty}(\overline{\mca{O}_{F_m}^1} ) ;
}$
\item $\xymatrix{ C_{F_m}^\times\otimes\bb{Z}_p & C^1\otimes\bb{Z}_p \ar[l]_(0.46){\simeq} \ar[r]^(0.56){\simeq}  & \overline{C_{F_m}^1} .
}$
\end{enumerate}
\end{rem}

\section{Fitting ideals}
Here, we recall the notion of higher Fitting ideals. 
\begin{dfn}[Higher Fitting ideals, see \cite{No} \S 3.1]
\label{Fitt}
Let $R$ be an commutative ring, and $M$ be a finitely presented $R$-module. Let
$$\xymatrix{
R^m \ar[r]^{f} & R^n \ar[r] & M \ar[r] & 0 \\
}$$ 
be an exact sequence of $R$-modules. For each $i \ge 0$, we define {\em the $i$-th Fitting ideal} $\Fitt_{R,i}(M)$ as  follows. When $0 \le i < n$ and $m \ge n-i$, we define $\Fitt_{R,i}(M)$ to be the ideal of $R$ generated by all $(n-i) \times (n-i)$ minors of the matrix corresponding to $f$. When $0 \le i < n$ and $m < n-i$, we define $\Fitt_{R,i}(M):=0$. When $i \ge n$, we define $\Fitt_{R,i}(M):=R$. The definition of these ideals depends only on $M$, and does not depend on the choice of the above exact sequence.
\end{dfn}

For a finitely presented $R$-module $M$, we have the following sequence of ideals of $R$:
$$\Fitt_{R,0}(M)\subseteq \Fitt_{R,1}(M)\subseteq \dots \subseteq \Fitt_{R,n}(M)= \Fitt_{R,n+1}(M)=\dots=R.$$

We denote the smallest number of generators of an $R$-module $M$ by $\nu_R(M)$. If $\Fitt_{R,n}(M) \neq R$, then $\nu_R(M) \ge n+1$. 
Note that when $R$ is a local ring or PID, we have  $\nu_R(M)=i+1$ if and only if $\Fitt_{R,i}(M)\neq R$ and $\Fitt_{R,i+1}(M)=R$.


\begin{exa}\label{rei-Iwasawa}
Let $R= \bb{Z}_p[[T]]$ and $M$ a finitely generated torsion $R$-module. 
Assume $$M\sim \bigoplus_{i=1}^nR/f_iR $$
and $f_i$ divides $f_{i+1}$ for $1 \le i \le n-1$. 
Then, for each $i$ with $i \ge 0$, there exists an ideal $I_i$ with finite index in $R$ such that 
$$\Fitt_{R,i}(M)=
\begin{cases}
(\prod_{k=1}^{n-i}f_k\big)I_i & (\text{if} \ i<n) \\
I_i & (\text{if} \ i \ge n)
\end{cases}
$$ (cf.\ \cite{Ku} Lemma 8.2). This implies that the family $\{\Fitt_{R,i}(M) \}_{i \ge 0}$ of Fitting ideals of $M$ determines the pseudo-isomorphism class of $M$. 
\end{exa}

We need the following lemma in the proof of Theorem \ref{Main theorem}.

\begin{lem}[for example, see \cite{Ku} Theorem 8.1]\label{no pn}
Let $R= \bb{Z}_p[[T]]$ and $M$ a finitely generated torsion $R$-module.
Suppose $M$ contains no non-trivial pseudo-null $R$-submodule. 
Then, there exists an exact sequence 
$$\xymatrix{0 \ar[r] & R^n \ar[r] & R^n \ar[r] & M \ar[r] & 0} $$
for some integer $n>0$, and we have 
$$\Fitt_{R,0}(M)=\cha_R(M).$$
\end{lem}

\section{The Euler system of cyclotomic units}

We first recall some basic results on the Euler system of cyclotomic units in \S\S \ref{saisho}-\ref{2-maps}. 
Then, in \S \ref{the subsection of x}, we define Kurihara's elements $x_{\nu,q} \in (F_m^{\times}/p^N)_{\chi}$, which play a key role in the proof of Theorem \ref{Main theorem} (cf.\ Definition \ref{x}). 
\subsection{}\label{saisho}
We fix a primitive $p^{m+1}$-st root of unity $\rho_m \in \bb{Q}(\mu_{p^{m+1}})$ for each $m \ge 0$ such that $\rho_{m+1}^p=\rho_m$, and a topological generator $e \in \bb{Z}_{>0}$ of $\bb{Z}_p^{\times}$. We have the following lemma.
\begin{lem}[\cite{Wa} Lemma 8.1]
\label{gene.cyc}
We define the element  $\mrm{cyc}(\rho_m)$ of $C_{F_m}$ by 
$$\mrm{cyc}(\rho_m):= \frac{\rho_m^{-e/2}-\rho_m^{e/2}}{\rho_m^{-1/2}-\rho_m^{1/2}}.$$
Then, the $\bb{Z}[\Gal(F_m/\bb{Q})]$-module $C_{F_m}$ is generated by $\pm \mrm{cyc}(\rho_m)$.
\end{lem}

\begin{cor}[\cite{CS} Lemma 4.3.4 ]
\label{gene.cyc.lim}
The $\Lambda$-module $\overline{C_{\infty}^1}$ is generated by $\big( u \cdot \mrm{cyc}(\rho_m) \big)_{m \ge 0}$, where $u$ is a $(p-1)$-st root of unity in $\bb{Q}_p$ such that $eu \equiv 1 \pmod{p}$. 
\end{cor}

For an integer $N \ge 1$, we define 
\begin{eqnarray*}
\mca{S}_N &:=& \big\{\ell \ | \ \ell \ \text{is a prime number not dividing} \ e , \ \text{and} \ \ell \equiv 1 \pmod{p^N}\big\}, \\
\mca{N}_N &:=& \big\{\prod_{i=1}^r \ell_i \ | \ r> 0, \ \ell_i \in \mca{S}_N \ (i=1,\dots,r), \ \text{and} \ \ell_i\not= \ell_j \ \text{if} \ i\not= j\big\} \cup\{1\}, \\
\end{eqnarray*}
and for any algebraic number field $K$, we define  
\begin{eqnarray*}
\mca{S}_N(K) &:=& \{\ell \in S_N | \ \ell \ \text{splits completely in} \ K/\bb{Q} \}, \\
\mca{N}_N(K) &:=& \big\{\prod_{i=1}^r \ell_i \ | \ r> 0, \ \ell_i \in \mca{S}_N(K) \ (i=1,\dots,r), \  \text{and} \ \ell_i\not= \ell_j \ \text{if} \ i\not= j\big\}\cup\{1\}. \\
\end{eqnarray*}
For $n=\prod_{i=1}^r \ell_i\in\mca{N}_N \ (\ell_i\in\mca{S}_N \ \text{for} \ i=1,\dots,r)$, we define $\epsilon(n):=r$.

\begin{dfn}
For $n\in\mca{N}_N$ and $\zeta\in\mu_{p^{m+1}n}\backslash \{ 1 \}$, we define  $$\mathrm{cyc}(\zeta):=\frac{\zeta^{-e/2}-\zeta^{e/2}}{\zeta^{-1/2}-\zeta^{1/2}} \in F_m(n),$$
where $F_m(n)$ denotes the maximal totally real subfield of $\bb{Q}(\mu_{p^{m+1}n})$. 
\end{dfn}

We obtain the following lemma immediately.
\begin{lem}
\label{norm}
Let $n \in \mca{N}_N$.
\begin{enumerate}
\item Let $\ell \in \mca{S}_N$, and assume $\ell$ does not divide $n$. Let $\zeta_\ell \in \mu_{\ell}$ be a primitive $\ell$-th root of unity, and $\xi \in \mu_{p^{m+1}n}\backslash \{ 1 \}$. Then $$N_{F_m(n\ell)/F_m(n)} \big{(} \mathrm{cyc}(\zeta_\ell\xi) \big{)} =\frac{\mathrm{cyc}(\xi^\ell)}{\mathrm{cyc}(\xi)}.$$
\item Let $\xi \in \mu_{n}$. Then 
$$N_{F_{m+1}(n)/F_m(n)}\big( \mathrm{cyc}(\rho_{m+1} \xi) \big) =\mathrm{cyc}(\rho_m \xi^p).$$
\end{enumerate}
\end{lem}

\subsection{}From now on, we assume $N\ge m+1$.
Let $n\in\mca{N}_N$. In this subsection, we shall define an element $\kappa(\xi) \in F_m^\times/p^N$ called Kolyvagin derivative for any primitive  $\xi\in\mu_n$. (In fact, we will define more general one. See Definition \ref{kappa}).

For an integer $n$ prime to $p$, we write $n=\prod_{i=1}^r \ell_i^{e_i}$ such that $\ell_1,\dots,\ell_r$ are distinct prime numbers and $e_i>0$ for each $i$. We define $F_m(n)$ to be  the maximal totally real subfield of $\bb{Q}(\mu_{p^{m+1}n})$, and $H_{F_m,n}:=\Gal\big( F_m(n)/F_m \big)$. Then, for any $m \ge 0$, we have canonical isomorphisms
\begin{align*}
H_{F_m,n} &= \Gal\big(F_m(n)/F_m \big)\simeq \Gal\big( \bb{Q}(\mu_{p^{m+1}n})/\bb{Q}(\mu_{p^{m+1}})\big) \\
 & \simeq  \Gal\big(\bb{Q}(\mu_n)/\bb{Q}\big)  \\
 & \simeq  \Gal \big(\bb{Q}(\mu_{\ell_1}^{e_1} )/\bb{Q}\big)\times \dots \times \Gal \big(\bb{Q}(\mu_{\ell_r}^{e_r})/\bb{Q} \big) \\
 & \simeq  H_{\ell_1^{e_1}}\times \dots \times H_{\ell_r^{e_r}}. 
\end{align*}
For all $m$ with $m\ge 0$, we identify $H_{F_m,n}$ with $H_{F_0,n}$ by canonical isomorphisms, and put $H_n:=H_{F_0,n}$. 

Recall $H_\ell$ is a cyclic group of order $\ell-1$ if $\ell$ is a prime number. We shall take a generator $\sigma_{\ell}$ of $H_\ell$ for each prime number $\ell \in \mca{S}_N$ as follows. Let $\ell\in \mca{S}_N$. We put $N_{\{\ell\} }:=\mrm{ord}_p(\ell-1)$, where $\mrm{ord}_p$ is the normalized additive valuation of $\ell$, namely, $\mrm{ord}_p(p)=1$. Then, we have $N_{\{\ell\} } \ge N \ge 1$.
By the fixed embedding $\ell_{\overline{\bb{Q}}}\colon\overline{\bb{Q}}\hookrightarrow \overline{\bb{Q}}_\ell$, we regard $\mu_{p^{N_{\{\ell\}}}}$ as a subset of $\bb{Q}_\ell$. We identify $\Gal \big(\bb{Q}_\ell(\mu_\ell)/\bb{Q}_\ell \big)$ with $H_\ell=\Gal \big(\bb{Q}(\mu_\ell)/\bb{Q} \big)$ by the canonical isomorphism. Let $K$ be the maximal $p$-extension field of $F_m$ contained in $F_m(\ell)$, and $\pi$ the prime element of $K_{\ell_K}$. We fix a generator $\sigma_\ell$ of $H_\ell$ such that $$\pi^{\sigma_\ell-1}\equiv \rho_{N_{\{\ell \}}-1} \pmod{\overline{\ell_K}},$$ where $\overline{\ell_K}$ is the maximal ideal of $K_{\ell_K}$, and $\rho_{N_{\{\ell\} }-1}$ is a primitive $p^{N_{\{\ell\}} }$-th root of unity defined as above. Note that the definition of $\sigma_\ell$ does not depend on the choice of $\pi$.

Let $n\in\mca{N}_N$. We define the following elements of the group ring $\bb{Z}[H_n]$.
\begin{dfn}
Let $n=\prod_{i=1}^r \ell_i \in\mca{N}_N$ such that $\ell_i\in\mca{S}_N$ for $i=1,\dots,r$. We define
$$D_{\ell_i}:=\sum_{k=1}^{\ell_i-2}k\sigma_{\ell_i}^k \in\bb{Z}[H_{\ell_i}] \subseteq \bb{Z}[H_{n}]$$ 
for $i=1,\dots,r$, and 
$$D_{n}:=\prod_{i=1}^{r}D_{\ell_i} \in \bb{Z}[H_{n}].$$
\end{dfn}

In order to define $\kappa(\xi)$, we need the following two well-known Lemmas.
\begin{lem}
\label{into}
Let $n_1,n_2 \in \mca{N}_N$ satisfying $(n_1,n_2)=1$.  
We put $n=n_1 n_2$. Then, the canonical map
$\xymatrix{
F_m(n_1)^{\times}/p^N \ar[r] & \big(F_m(n)^{\times}/p^N\big)^{H_{n_2}} \\
}$
is isomorphism.
\end{lem}


\begin{lem}
\label{fix}
Let $n_1,n_2 \in \mca{N}_N$. Assume $\ell \in \mca{S}_N\big( F_m(n_1) \big)$ for each prime divisor $\ell$ of $n_2$. 
Namely, $\ell \equiv 1 \pmod{p^N n_1}$ for each prime divisor $\ell$ of $n_2$. 
We put $n=n_1 n_2$. Let $\xi_{n_1} \in\mu_{n_1}$ be a primitive $n_1$-st root of unity, and $\xi_{n_2} \in\mu_{n_2}$ a primitive $n_2$-nd root of unity. 
Then, the image of $\mathrm{cyc}(\rho_m\xi_{n_1}\xi_{n_2})^{D_{n_2}}$ in $F_m(n)^{\times}/p^N$ is fixed by $H_{n_2}=\Gal\big( F_m(n)/F_m(n_1) \big)$. 
\end{lem}

\begin{dfn}\label{kappa}
Let $n_1,n_2 \in \mca{N}_N$. Assume $\ell \in \mca{S}_N\big( F_m(n_1) \big)$ for each prime divisor of $n_2$.
We put $n=n_1 n_2$. Let $\xi\in\mu_n$ be a primitive $n$-th root of unity. 
We define $$\kappa^{n_1}_{m,N}(\xi) \in F_m(n_1)^\times/p^N $$ to be the unique element of $F_m(n_1)^\times/p^N$ such that its image in $F_m(n)^{\times}/p^N$ is the class of $\mathrm{cyc}(\rho_m\xi)^{D_{n_2}}$.
When no confusion arises, we denote $\kappa^{n_1}_{m,N}(\xi)$ by $\kappa^{n_1}(\xi)$ for simplicity.

When $n_1=1$, the element $\kappa^{1}(\xi) \in F_m^\times/p^N$ is denoted by $\kappa(\xi)$.
\end{dfn}

\subsection{}\label{2-maps}
Let $R_{F_m,N}:=\bb{Z}/p^N[\Gal(F_m/\bb{Q})]$.  Let $\chi \in \widehat\Delta$ be a character. We define $R_{F_m,N,\chi}:=\bb{Z}/p^N[\Gal(F_m)/\bb{Q}]_\chi$ to be the $\chi$-part of $R_{F_m,N}$. Namely, $R_{F_m,N,\chi}$ is the ring isomorphic to $\bb{Z}/p^N[\Gal(F_m/F_0)]$ on which $\Delta$ acts via $\chi$. Obviously, we have 
$$
R_{F_m,N} = \Lambda_{\Gamma_m}/p^N, \ \
R_{F_m,N,\chi} = \Lambda_{\chi, \Gamma_m}/p^N\ \  \text{and}\ \  
R_{F_m,N} = \bigoplus_{\chi \in \widehat{\Delta}} R_{F_m,N,\chi},
$$
where $\Lambda_{\Gamma_m}$ (resp.\ $\Lambda_{\chi,\Gamma_m}$) denotes the $\Gamma_m$-coinvariant of $\Lambda$ (resp.\ $\Lambda_\chi$).
For a $R_{F_m,N}$-module $M$ and an element $x \in M$, we denote the $\chi$-component of $x$ by $x_\chi\in M_\chi$.

As in \cite{Ku} \S 2.3, in this subsection, we shall define two homomorphisms 
$$\xymatrix{
[\cdot ]_{F_m,N,\chi}^\ell\colon (F_m^{\times}/p^N)_{\chi} \ar[r] & R_{F_m, N, \chi} \\
}$$ 
for each $\ell\in\mca{S}_N$ (cf.\ Definition \ref{[]}), and 
$$\xymatrix{
\bar{\phi}_{F_m(n),N,\chi}^\ell\colon (F_m(n)^{\times}/p^N)_{\chi} \ar[r] & R_{F_m, N, \chi}[H_n] \\
}$$ 
for each $n \in \mca{N}_N$ and $\ell\in\mca{S}_N\big( F_m(n) \big)$ (cf.\ Definition \ref{phi}). 
The homomorphism $[\cdot ]_{F_m,N,\chi}^\ell$ is defined by the valuations of the places above $\ell$, and $\bar{\phi}_{F_m(n),N,\chi}^\ell$ is defined by the local reciprocity maps. 

First, we define $[\cdot ]_{F_m,N,\chi}^\ell$. Let $K$ be an algebraic number field. We define $\mca{I}_K$ to be the group of  fractional ideals of $K$, and we write its group law additively.
We define the homomorphism $\xymatrix{(\cdot )_K\colon K^{\times} \ar[r] & \mca{I}_K}$ by $$(x)_K=\sum_{\lambda}\mrm{ord}_{\lambda}(x)\lambda,$$ where $\lambda$ runs through all prime ideals of $K$, and $\mrm{ord}_{\lambda}$ is the normalized valuation of $\lambda$. For any prime number $\ell$, we define  $\mca{I}_K^\ell$ to be the subgroup of $\mca{I}_K$ generated by all prime ideals above $\ell$. Then, we define $\xymatrix{(\cdot )_K^\ell\colon K^{\times} \ar[r] & \mca{I}_K^\ell}$ by$$(x)_K^\ell=\sum_{\lambda|\ell}\mrm{ord}_{\lambda}(x)\lambda.$$

Recall that we fix a family of embeddings $\{ \xymatrix{\ell_{\overline{\bb{Q}}}\colon\overline{\bb{Q}}\ar@{^{(}->}[r] & \overline{\bb{Q}}_\ell} \}_{\ell: \rm{prime}}$ (cf. \S 1 Notation), and we denote  the ideal of $K$ corresponding to the  embedding $\ell_{\overline{\bb{Q}}}|_K$ by $\ell_K$ for each prime number $\ell$ and algebraic number field $K$. Assume $\ell\in\mca{S}_N(K)$ and $K/\bb{Q}$ is Galois extension. Then, $\mca{I}_K^\ell$ is a free $\bb{Z}[\Gal(K/\bb{Q})]$-module generated by $\ell_K$, and we identify $\mca{I}_K^\ell$ with $\bb{Z}[\Gal(K/\bb{Q})]$ by the isomorphism $\xymatrix{\iota\colon \bb{Z}[\Gal(K/\bb{Q})] \ar[r]^(0.68){\simeq} & \mca{I}_K^\ell}$ defined by ${x \longmapsto x\cdot \ell_K}$ for $x \in \bb{Z}[\Gal(K/\bb{Q})]$.
The composition 
$\xymatrix{
K^{\times} \ar[r] & \mca{I}_K^\ell \ar[r]^(0.3){\iota^{-1}}& \bb{Z}[\Gal(K/\bb{Q})]  
}$
is also denoted by $(\cdot )_K^\ell$.

\begin{dfn}
\label{[]}
We define the $R_{F_m, N, \chi}$-homomorphism $$\xymatrix{
[\cdot ]_{F_m,N,\chi}\colon (F_m^{\times}/p^N)_{\chi} \ar[r] & (\mca{I}_{F_m}/p^N)_\chi
}$$ 
to be the homomorphism induced by $\xymatrix{(\cdot )_{F_m}^\ell\colon F_m^{\times} \ar[r] & \mca{I}_{F_m}}$.

Let $\ell\in\mca{S}_N$. We define the $R_{F_m, N, \chi}$-homomorphism
$$\xymatrix{
[\cdot ]_{F_m,N,\chi}^\ell\colon (F_m^{\times}/p^N)_{\chi} \ar[r] & R_{F_m, N, \chi}=\bb{Z}/p^N[\Gal(F_m/\bb{Q})]_{\chi}
}$$ 
to be the homomorphism induced by $\xymatrix{(\cdot )_{F_m}^\ell\colon F_m^{\times} \ar[r] & \bb{Z}[\Gal(F_m/\bb{Q})]}$.
\end{dfn}

Second, we will define $\bar{\phi}_{F_m(n),N,\chi}^\ell$. Let $n \in \mca{N}_N,$ and $\ell\in\mca{S}_N(F_m(n))$. Since we assume $N \ge m+1$, the prime number $\ell$ splits completely in $F_m(n)/\bb{Q}$, and we have $F_m(n)_{\lambda}=\bb{Q}_\ell$ for any prime ideal $\lambda$ of $F_m(n)$ above $\ell$. 
We regard the group $\bb{Q}_\ell^\times$ as a $\bb{Z}[\Gal(F_m/\bb{Q})]$-module with the trivial action of $\Gal(F_m/\bb{Q})$, and groups  $\bigoplus_{\lambda|\ell}F_m(n)_{\lambda}^{\times}$ and $\bigoplus_{\lambda|\ell}H_\ell$ are regarded as $\bb{Z}[\Gal(F_m/\bb{Q})]$-modules by the identification 
$$\bigoplus_{\lambda|\ell}F_m(n)_{\lambda}^{\times} = \mca{I}_{F_m(n)}^\ell\otimes_{\bb{Z}} \bb{Q}_\ell^\times \ \ \text{and}\ \ 
\bigoplus_{\lambda|\ell}H_\ell = \mca{I}_{F_m(n)}^\ell\otimes H_\ell,
$$
respectively.

We denote by $$\xymatrix{\phi_{\bb{Q}_\ell}\colon \bb{Q}_\ell^{\times} \ar[r] & \Gal\big(\bb{Q}_\ell(\mu_\ell)/\bb{Q}_\ell\big)=\Gal\big(\bb{Q}(\mu_\ell)/\bb{Q}\big) = H_\ell \\}$$ the reciprocity map of local class field theory defined by $\phi_{\bb{Q}_\ell}(\ell)=(\ell_{\bb{Q}(\mu_\ell)},\bb{Q}(\mu_\ell)/\bb{Q})$. The homomorphism  
$$\xymatrix{
\phi_{F_m(n)}^{\ \ell}\colon F_m(n)^{\times} \ar[r] & \bb{Z}[\Gal(F_m(n)/\bb{Q})]\otimes H_\ell \\ 
}$$ 
is defined to be the composition of the three homomorphisms of $\bb{Z}[\Gal( F_m(n)/\bb{Q} )]$-modules:
$$\xymatrix{\mrm{diag}\colon F_m(n)^{\times} \ar[r] & \bigoplus_{\lambda|\ell}F_m(n)_{\lambda}^{\times}, }$$
$$\xymatrix{\oplus\phi_{\bb{Q}_\ell} \colon \bigoplus_{\lambda|\ell}F_m(n)_{\lambda}^{\times} \ar[r] & \bigoplus_{\lambda|\ell}H_\ell, }$$
$$\xymatrix{\iota_H^{-1} \colon \bigoplus_{\lambda|\ell}H_\ell \ar[r]^(0.4){\simeq } & \bb{Z}[\Gal( F_m(n)/\bb{Q})]\otimes H_\ell, }$$
which are defined  as follows:
\begin{enumerate}
\item the first homomorphism $\mrm{diag}$ is the diagonal inclusion;
\item the second homomorphism $\oplus\phi_{\bb{Q}_\ell}$ is the direct sum of the reciprocity maps; 
\item the third isomorphism $\iota_H^{-1}$ is the inverse of the isomorphism 
$$\xymatrix{\iota_H \colon \bb{Z}[\Gal( F_m(n)/\bb{Q}) ]\otimes H_\ell \ar[r]^(0.52){\simeq} & \bigoplus_{\lambda|\ell}H_\ell=\mca{I}_{F_m(n)}^\ell\otimes H_\ell, \\}$$
which is induced by the above isomorphism 
$$\xymatrix{\iota\colon \bb{Z}[\Gal( F_m(n)/\bb{Q})] \ar[r]^(0.64){\simeq} & \mca{I}_{F_m(n)}^\ell}$$
given by 
$x \longmapsto x \cdot \ell_{F_m(n)}$.
\end{enumerate}

\begin{dfn}
\label{phi}Let $n \in \mca{N}_N,$ and $\ell\in\mca{S}_N(F_m(n))$.
We define $$\xymatrix{\phi_{F_m(n),N,\chi}^\ell\colon (F_m(n)^{\times}/p^N)_{\chi} \ar[r] & \bb{Z}/p^N[\Gal(F_m(n)/\bb{Q})]_{\chi}\otimes H_\ell \\}$$ to be the homomorphism of  $R_{F_m,N,\chi}[H_n]$-modules induced by $\phi_{F_m(n)}^\ell$. 
The choice of a generator $\sigma_\ell$ of $H_\ell$ induces the $R_{F_m,N,\chi}[H_n]$-homomorphism
$$\xymatrix{\bar{\phi}_{F(n)_m,N,\chi}^\ell\colon (F(n)_m^{\times}/p^N)_{\chi} \ar[r] & \bb{Z}[\Gal(F_m(n)/\bb{Q})]_{\chi}= R_{F_m,N,\chi}[H_n] \\}.$$ 
\end{dfn}

Next, we shall prove some formulas of the Euler system of cyclotomic units.
We fix a primitive $\ell$-th root of unity $\xi_\ell$ for each $\ell \in \mca{S}_N$. For each $n \in \mca{N}_N$, we define a primitive $n$-th root of unity by $$\xi_n=\prod_{\substack{ \ell \in \mca{S}_N  \\ \ell |n}} \xi_\ell.$$ 
As in \cite{Ku}, we use the notion {\em well-ordered}.
\begin{dfn}
\label{well-ordered}
Let $n\in\mca{N}_N$. 
We call $n$ {\em well-ordered} if and only if $n$ has a factorization $n=\prod_{i=1}^r \ell_i$ such that $\ell_{i+1} \in S(F_m(\prod_{j=1}^i \ell_j))$ for$i=1,\dots,r$. In other words, $n$ is well-ordered if and only if $n$ has a factorization $n=\prod_{i=1}^r \ell_i$ such that $$\ell_{i+1} \equiv 1 \pmod{p^N \prod_{j=1}^i \ell_j}$$ for $i=1,\dots,r-1$.
\end{dfn}

\begin{prop}
\label{[] and phi}
Let $n$ be an integer contained in $\mca{N}_N$. 
\begin{enumerate}
\item  If $\lambda$ is a prime ideal of $K$ not dividing $n$, the $\lambda$-component of $[\kappa(\xi_n)_\chi]_{F_m,N,\chi}$ is $0$. 
In particular, if $q \in \mca{S}_N$ is a prime number not dividing $n$, we have $$[\kappa(\xi_n)_\chi]_{F_m,N,\chi}^q = 0.$$
\item Let $\ell$ be a prime number dividing $n$. Then,  $$[\kappa(\xi_n)_\chi]_{F_m,N,\chi}^{\ell}=-\bar{\phi}_{F_m,N,\chi}^{\ell}(\kappa(\xi_{n/\ell})_\chi).$$ 
\item If $n$ is well-ordered, then $$\bar{\phi}_{F_m,N,\chi}^{\ell}(\kappa(\xi_n)_\chi)=0 $$ for each prime number $\ell$ dividing $n$. {\em (See \cite{MR} Theorem A.4.)}
\end{enumerate}
\end{prop} 

\begin{rem}
The third assertion of Proposition \ref{[] and phi} is cyclotomic unit version of Lemma 5.3 in \cite{Ku}, and also a special case of \cite{MR} Theorem A.4. 
In this paper, using the argument in \cite{Ku}, we shall directly give a proof of Proposition \ref{[] and phi} (3). 
\end{rem}

\begin{proof}
Let us prove Proposition \ref{[] and phi}.

We prove the first assertion. If $\lambda$ is a prime ideal of $F_m$ not dividing $pn$, the $\lambda$-component of $[\kappa(\xi_n)_\chi]_{F_m,N,\chi}$ is $0$ since $F_m(n)/F_m$ is unramified outside $pn$. Let $\mf{p}_m$ be the (unique) prime ideal of $F_m$ above $p$, and $\mf{P}$ a prime ideal of $F_m(n)$ above $\mf{p}_m$. The ramification index of $\mf{P}/\mf{p}$ is $2$. Since $p\not=2$, the $\mf{p}_m$-component of $[\kappa(\xi_n)_\chi]_{F_m,N,\chi}$ is $0$. The proof of the first assertion is complete. 

The second assertions is \cite{CS} Theorem 5.4.9. Note that $l_\ell:= -\bar{\phi}_{F_m,N,\chi}^{\ell}$ is used in \cite{CS} instead of our $\bar{\phi}_{F_m,N,\chi}^{\ell}$.

We shall prove the third assertion.  
Assume $n=\prod_{i=1}^r \ell_i \in\mca{N}_N$, where $\ell_1,\dots, \ell_r$ are distinct prime numbers, and $\ell_{i+1} \equiv 1 \pmod{p^N \prod_{j=1}^i \ell_j}$ for each $i=1,\dots,r-1$.
We put $n_1:=\prod_{j=1}^{i-1}\ell_j$. 
Note $\ell_1 \in \mca{S}_N(F_m(n))$.
It is sufficient to prove the following claim.
\begin{claim}\label{relative}
$\bar{\phi}_{F_m(n_1),N,\chi}^{\ell_i}\big( \kappa^{n_1}(\xi_n) \big)=0.$
\end{claim}
Let $\lambda$ be a prime ideal of $F_m(n_1)$ above $\ell_i$, and $\lambda'$ the prime of $F_m(n_1\ell_i)$ above $\lambda$.
Let $\pi'$ be a prime element of $F_m(n_1\ell_i)_{\lambda'}$. We take the prime element $\pi$ of $F_m(n_1)_\lambda$ defined by $$\pi:=N_{F_m(n_1\ell_i)_{\lambda'}/F_m(n_1)_{\lambda}}(\pi').$$

We have a decomposition $F_m(n_1)_{\lambda}^\times/p^N = U \times P$
as a group, where $U$ is the image of the unit group of $F_m(n_1\ell_i)_{\lambda}$ and $P$ is the cyclic subgroup generated by the image of $\pi$. Both group $U$ and $P$ are cyclic groups of order $p^N$. The group $U$ is generated by the image of a $p^{N_{\ell_i}}$-th root of unity in $F_m(n_1)_{\lambda}$.

Similarly, we have the decomposition $F_m(n_1\ell_i)_{\lambda'}^\times/p^N = U' \times P'$
as a group, where $U'$ is the image of the unit group of $F_m(n_1\ell_i)_{\lambda'}$ and $P'$ is the cyclic subgroup generated by the image of $\pi'$.
Both group $U'$ and $P'$ are cyclic groups of order $p^N$ with the action of $H_{\ell_i}$. 

We consider the homomorphism 
$\xymatrix{\phi_{\bb{Q}_{\ell},N}\colon F_m(n_1)_{\lambda}^\times/p^N  = \bb{Q}_\ell^{\times}/p^N  \ar[r] & H_\ell \otimes \bb{Z}/p^N\bb{Z}}$
induced by $\phi_{\bb{Q}_{\ell}}$. 
By local class field theory, the kernel of $\phi_{\bb{Q}_{\ell},N}$ is $P$. 
On the other hand, the kernel of the natural homomorphism $\xymatrix{\iota\colon F_m(n_1)_{\lambda}^\times/p^N \ar[r] & F_m(n_1\ell_i)_{\lambda'}^\times/p^N  }$
is also $P$. Then, we have $\ker\phi_{\bb{Q}_{\ell},N}=\ker\iota$. 

We shall show that $\iota\big(\kappa^{n_1}(\xi_n)\big)=1$. 
Let $\lambda''$ be a prime ideal of $F_m(n)$ above $\lambda$.
Since $F_m(n)_{\lambda''}/F_m(n_1)_{\lambda}$ is unramified and $\mathrm{cyc}(\rho_m\xi_n)$ is a unit in $F_m(n)_{\lambda''}$,
the element $\kappa^{n_1\ell_i}(\xi_n)$ is contained in $U'$. 
Then, the action of $H_{\ell_i}$ on $\kappa^{n_1\ell_i}(\xi_n)^{D_{\ell_i}} \in F_m(n_1\ell_i)_{\lambda'}/p^N$ is trivial, and we have 
\begin{align*}
\iota\big(\kappa^{n_1}(\xi_n)\big) &= \kappa^{n_1\ell_i}(\xi_n)^{D_{\ell_i}} \\
 &= \kappa^{n_1\ell_i}(\xi_n)^{\sum_{k=1}^{\ell_i-2}k\sigma_{\ell_i}^k} \\
 &= \kappa^{n_1\ell_i}(\xi_n)^{(\ell_i-1)(\ell_i-2)/2}  \\
 &= 1
\end{align*}
in $F_m(n_1\ell_i)_{\lambda'}^\times/p^N$. 
The proof of Claim \ref{relative} and Proposition \ref{[] and phi} (3) is complete.
\end{proof}

\subsection{}\label{the subsection of x}
In this subsection, we will define the Kurihara's elements $x_{\nu ,q} \in (F_m^{\times}/p^N)_{\chi}$ which become a key of the proof of Theorem \ref{Main theorem}.

\begin{dfn}
Let $q\nu =q\prod_{i=1}^r \ell_i \in\mca{N}_N$, where $q, \ell_1,\dots,\ell_r$ are distinct prime numbers. For a positive integer $d$ dividing $\nu$, we define $\tilde{\kappa}_{d,q} \in (F_m^\times/p^N)_{\chi} \otimes(\bigotimes_{\ell |d} H_\ell)$ by$$\tilde{\kappa}_{d,q} := \kappa(\xi_q\prod_{\ell |d}\xi_\ell)_\chi\otimes\big( \bigotimes_{\ell |d}\sigma_\ell \big).$$ 
\end{dfn}

Let $q\nu \in\mca{N}_N$ and assume $q\nu$ is {\em well-ordered}. Assume that for each prime number $\ell$ dividing $\nu$, an element $a_\ell \in R_{F_m,N, \chi} \otimes H_\ell$ is given. Then, we have an element $\bar{a}_\ell \in R_{F_m, N, \chi}$ such that $a_\ell = \bar{a}_\ell\otimes \sigma_\ell$. Note that we will take $\{ a_\ell \}_{\ell | \nu}$ explicitly later, but here, we take arbitrary one.

For a positive integer $d$ dividing $\nu$, we define the element $a_d$ by 
$$a_d:=\bigotimes_{\ell |d} a_\ell \in R_{F_m,N, \chi} \otimes \big (\bigotimes_{\ell |d} H_\ell \big),$$
and the element $\bar{a}_d \in R_{F_m,N, \chi}$ by 
$$a_d = \bar{a}_d \otimes\big( \bigotimes_{\ell |d}\sigma_\ell \big).$$

Note that we write the group law of $(F_m^{\times}/p^N)_{\chi} \otimes \big( \bigotimes_{\ell |d} H_\ell \big)$ multiplicatively. 
\begin{dfn}
\label{x}
We define the element  $\tilde{x}_{\nu, q} $ by$$\tilde{x}_{\nu ,q} := \prod_{d|\nu}a_d\otimes \tilde{\kappa}_{\nu/d,q} \in (F_m^{\times}/p^N)_{\chi} \otimes \big( \bigotimes_{\ell |d} H_\ell \big).$$
Note that we naturally identify the $R_{F_m, N, \chi}$-module $(F_m^{\times}/p^N)_{\chi} \otimes \big( \bigotimes_{\ell |d} H_\ell \big)$ with $$(F_m^{\times}/p^N)_{\chi}\otimes_{R_{F_m, N, \chi}}R_{F_m, N, \chi} \otimes \big( \bigotimes_{\ell |d} H_\ell \big).$$
The element $x_{\nu,q} \in (F_m^{\times}/p^N)_{\chi}$ is defined by 
$$\tilde{x}_{\nu,q}=x_{\nu,q}\otimes\big( \bigotimes_{\ell | \nu}\sigma_\ell \big).$$
\end{dfn}

The following formulas follows from Proposition \ref{[] and  phi} easily.

\begin{prop}[cf. \cite{Ku} Proposition 6.1]\label{[] and phi and x}
Let $q\nu \in  \mca{N}_N$ and we assume that $q\nu$ is well-ordered.
\begin{enumerate}
\item If $\lambda$ is a prime ideal of $K$ not dividing $n$, the $\lambda$-component of $[x_{\nu,q}]_{F_m,N,\chi}$ is $0$. 
In particular, if $s$ is a prime number not dividing $q\nu$, we have $$[x_{\nu,q}]_{F_m,N,\chi}^s = 0.$$
\item Let $\ell$ be a prime number dividing $\nu$. Then, we have   $$[x_{\nu,q}]_{F_m,N,\chi}^{\ell}=-\bar{\phi}_{F_m,N,\chi}^{\ell}(x_{\nu/\ell,q}).$$
\item Let $\ell$ be a prime number dividing $\nu$. Then, we have $$\bar{\phi}_{F_m,N,\chi}^{\ell}(x_{\nu,q})= \bar{a}_\ell\bar{\phi}_{F_m,N,\chi}^{\ell}(x_{\nu/\ell,q}).$$
\end{enumerate}
\end{prop}
%
%
%

\section{An application of the Chebotarev density theorem}\label{the section of chebo}

Recall that we fix a family of embeddings $\{ \xymatrix{\ell_{\overline{\bb{Q}}}\colon\overline{\bb{Q}}\ar@{^{(}->}[r] & \overline{\bb{Q}}_\ell} \}_{\ell: \rm{prime}}$ satisfying the technical condition (A) for families of embeddings as follows.
\begin{itemize}
\item[(A)]  {\em For any subfield $K \subset \overline{\bb{Q}}$ which is a finite Galois extension of $\bb{Q}$ and any element $\sigma \in \Gal(K/\bb{Q})$, there exist infinitely many prime numbers $\ell$ such that $\ell$ is unramified in $K/\bb{Q}$ and $(\ell_K, K/\bb{Q})= \sigma$, where $\ell_K$ is the prime ideal of $K$ corresponding to the  embedding $\ell_{\overline{\bb{Q}}}|_K$.}
\end{itemize}
Note that the existence of such a family of embeddings follows from the Chebotarev density theorem. 
We need the condition (A) in the proof of Proposition \ref{chebo appli}.

Here, we shall prove Proposition \ref{chebo appli}, which is the key of induction argument in the proof of Theorem \ref{Main theorem}.
This proposition corresponds to Lemma 9.1 in \cite{Ku}.
\begin{prop}
\label{chebo appli}
Let $\chi \in \widehat\Delta $ be a non-trivial character. 
Assume $qn=q\prod_{i=1}^r\ell_i \in \mca{N}_N$, where $q,\ell_1,\dots,\ell_r$ are prime numbers. Suppose the following are given:
\begin{itemize}
\item an element $\tau_i \in R_{F_m,N,\chi} \otimes H_{\ell_i}$ for each $i=1,\dots,r$;
\item a finite $R_{F_m, N, \chi}$-submodule $W$ of $(F_m^{\times}/p^N)_{\chi}$;
\item a $R_{F_m, N, \chi}$-homomorphism $\xymatrix{\lambda\colon W \ar[r] & R_{F_m, N, \chi}.}$
\end{itemize}
Then, there exist infinitely many $q' \in \mca{S}_N(F_m(n))$ which have the following properties:
\begin{enumerate}
\item the class of $q'_{F_m}$ in $A_{m,\chi}$ coincides with the class of $q_{F_m}$;
\item there exists an element $z \in (F_m^\times\otimes \bb{Z}_p)_{\chi}$ such that $$(z)_{F_m,\chi}=(q'_{F_m}-q_{F_m})_\chi \in \big(\mca{I}_{F_m}\otimes\bb{Z}_p \big)_\chi,$$ and $$\phi_{F_m,N,\chi}^{\ell_i}(z)=\tau_i$$ for each $i=1,\dots,r$;
\item the group $W$ is contained in the kernel of $[\cdot ]_{F_m,N,\chi}^{q'}$, and $$\lambda(x)=\bar{\phi}_{F_m,N,\chi}^{q'}(x)$$ for any $x \in W$.
\end{enumerate}
\end{prop}

\begin{proof}
We shall prove Proposition \ref{chebo appli} by four steps, using argument as in \cite{Ku}.

{\em The first step. Here, we first define a field $F_m\{n\}_\chi$, which is finite Galois over $F_m$. Then, we take an element $\sigma \in \Gal\big(K/\bb{Q}(\mu_{np^N})\big)$ corresponding to the conditions(1) and (2) of Proposition \ref{chebo appli}, where $K:=F_m\{n\}_\chi \bb{Q}(\mu_{np^N})$ is the composition field.}

Let $v$ be a prime ideal of $F_m$. We denote the ring of integers of the completion $F_{m,v}$ of $F_m$ at $v$ by $\mca{O}_{F_{m,v}}$, and define the subgroup $\mca{O}_{F_{m,v}}^1$ of $\mca{O}_{F_{m,v}}^\times$ by 
$$\mca{O}_{F_{m,v}}^1:=\{x \ |\ x \equiv 1 \pmod{\bar{v}} \},$$
where $\bar{v}$ is the maximal ideal of $\mca{O}_{F_{m,v}}$.
We denote the residue field of $F_m$ at $v$ by $k(v)$.

Let $F_m\{n \}$ be the maximal abelian $p$-extension of $F_m$ unramified outside $n$. By global class field theory, we have the isomorphism
$$\xymatrix{\displaystyle \frac{(\prod_{v|n}F_{m,v}^\times/\mca{O}_{F_{m,v}}^1 )\times (\bigoplus_{u \nmid n}F_{m,u}^\times/\mca{O}_{F_{m,u}}^\times)}{ F_m^\times}\otimes\bb{Z}_p \ar[r]^(0.68){\simeq} & \Gal\big(F_m\{n\}/F_m\big),}$$
where $u$ runs all finite places outside $n$.
This isomorphism induces the homomorphism
$$\xymatrix{\iota\colon\bigoplus_{v|n}k(v)^\times \otimes \bb{Z}_p  \ar[r]& \Gal\big(F_m\{n\}/F_m\big). }$$
Taking the $\chi$-part, we obtain the homomorphism 
$$\xymatrix{\iota_\chi\colon\big(\bigoplus_{v|n} k(v)^\times \otimes \bb{Z}_p \big)_{\chi}  \ar[r] & \Gal\big(F_m\{n\}/F_m\big)_{\chi}}$$
of $\bb{Z}_p[\Gal(F_m/\bb{Q})]_\chi$-modules.
We denote by $F_m\{n\}_\chi$ the intermediate field of $F_m\{n\}/F_m$ with $\Gal\big(F_m\{n\}_\chi/F_m\big)=\Gal(F_m\{n\}/F_m)_\chi$.

Recall $n=\prod_{i=1}^r\ell_i$, and all prime divisors $\ell_i$ of $n$ split completely in $F_m/\bb{Q}$. By local Artin maps, we obtain the isomorphism
$$\xymatrix{\big(\bigoplus_{v|n}k(v)^\times \otimes \bb{Z}_p \big)_{\chi}  \ar[r]^(0.37){\simeq} & \bigoplus_{\i=1}^r\big(\bb{Z}_p[\Gal(F_m/\bb{Q})]_{\chi} \cdot \ell_{i,F_m} \big)\otimes H_{\ell_i}, }$$
and we identify them by this isomorphism. We take an element $$\tau_\chi=(\tau_{\chi,v})_{v|n} \in \big(\bigoplus_{v|n}k(v)^\times \otimes \bb{Z}_p \big)_{\chi}$$ whose image in $\bigoplus_{\i=1}^r(R_{F_m,N,\chi} \cdot \ell_{i,F_m} )\otimes H_{\ell_i}$ is $(\tau_1\cdot\ell_{1,F_m},\dots,\tau_r\cdot\ell_{r,F_m})$.

Let $K:=F_m\{n\}_\chi \bb{Q}(\mu_{p^Nn})$ be the composition field. Since the subgroup $\Delta$ of $\Gal(F_m/\bb{Q})$ acts on $\Gal\big(F_m\{n\}_{\chi}/F_m\big)$ (resp.\ $\Gal\big(F_{N-1}(n)/F_m\big)$) via $\chi$ (resp.\ trivial character) and $\chi$ is non-trivial, we have 
$$F_m\{n\}_{\chi}\cap \bb{Q}(\mu_{np^N}) =F_m\{n\}_{\chi}\cap F_{N-1}(n)=F_m.$$ 
Then, we take the element $\sigma \in \Gal\big(K/\bb{Q}(\mu_{np^N})\big)$ such that 
$$\sigma |_{F_m\{n\}_\chi}=\iota_\chi(\tau_\chi)^{-1}(q_{F_m\{n\}_\chi}, F_m\{n\}_\chi/F_m ).$$

{\em The second step. 
Here, we shall take an element $\lambda' \in \Gal\big(L/\bb{Q}(\mu_{p^N})\big)$ corresponding to the condition (3) of Proposition \ref{chebo appli}, where $L$ is the extension field of $\bb{Q}(\mu_{p^N})$ generated by all $p^N$-th roots of elements contained in $W$}. 

We define a projection $\xymatrix{\mrm{pr}\colon R_{F_m,N} \ar[r] & \bb{Z}/p^N\bb{Z}}$ by 
$$\sum_{g \in \Gal(F_m/\bb{Q})}a_g g \longmapsto  a_1,$$
where $a_g \in \bb{Z}/p^N\bb{Z}$ for all $g \in \Gal(F_m/\bb{Q})$, and $1 \in \Gal(F_m/\bb{Q})$ is the unit. 
We define $\lambda'\in \Hom(W,\mu_{p^N})$ by $x \longmapsto \rho _{N-1}^{\mrm{pr}\circ \lambda(x)}$
for all $x \in W$. (Recall $\rho_{N-1}$ is a primitive $p^N$-th root of unity defined in \S \ref{saisho}.)
We use the following well-known lemma.
\begin{lem}\label{onajimono}
Let 
$\xymatrix{P\colon \Hom_{R_{{F_m},N,\chi}}(W,R_{{F_m},N,\chi}) \ar[r] &  \Hom(W,\bb{Z}/p^N\bb{Z})}$ be the map given by
$f \longmapsto \mrm{pr}\circ f $. 
Then, $P$ is bijective.
\end{lem}
Indeed, the inverse of $P$ is given by 
$$h \longmapsto \bigg(x\longmapsto \sum_{g \in\Gal(F_m/\bb{Q})}h(g^{-1}x)g \bigg) \in \Hom_{R_{{F_m},N,\chi}}(W,R_{{F_m},N,\chi}),$$
for $h \in \Hom(W,\bb{Z}/p^N\bb{Z})$. Note that $\Delta$ acts on $W$ via $\chi$, so we have  $$\Hom_{R_{{F_m},N}}(W,R_{{F_m},N})=\Hom_{R_{{F_m},N,\chi}}(W,R_{{F_m},N,\chi}).$$ 

The natural homomorphism 
$$\xymatrix{W \subset (F_m^\times/p^N)_{\chi} \ar[r] & \big(\bb{Q}(\mu_{p^N})^\times/p^N}\big)_{\chi}$$ is injective since we have 
$H^1\big(\Gal\big(\bb{Q}(\mu_{p^N})/F_m\big), \mu_{p^N}\big)_\chi=0$. 
So, we regard $W$ as a subgroup of $\big(\bb{Q}(\mu_{p^N})^\times/p^N\big)_{\chi}$.
Let $L$ be the extension field of $\bb{Q}(\mu_{p^N})$ generated by all $p^N$-th roots of elements of $F_m^\times$ whose image in $F_m^\times/p^N$ is contained in $W$. 
We consider the Kummer pairing $$\xymatrix{\Gal\big(L/\bb{Q}(\mu_{p^N})\big)\times W \ar[r] & \mu_{p^N}.}$$
This pairing induces the isomorphism $\Hom(W,\mu_{p^N})\simeq \Gal\big(L/\bb{Q}(\mu_{p^N})\big)$. We regard $\lambda'$ as an element of $\Gal\big(L/\bb{Q}(\mu_{p^N})\big)$ by this isomorphism.

{\em The third step. 
Here, we show that $K \cap L= \bb{Q}(\mu_{p^N})$}. 

Since we have the natural isomorphism $$\Gal\big(\bb{Q}(\mu_{p^N})/\bb{Q}\big)\simeq \Gal\big(\bb{Q}(\mu_p)/\bb{Q}\big)\times \Gal\big(\bb{Q}(\mu_{p^N})/\bb{Q}(\mu_p)\big),$$ we regard $\Gal\big(\bb{Q}(\mu_p)/\bb{Q}\big)$ as a subgroup of $\Gal\big(\bb{Q}(\mu_{p^N})/\bb{Q}\big)$. 
Let $c \in \Gal\big(\bb{Q}(\mu_{p})/\bb{Q}\big)$ be the complex conjugation. Namely, $c$ is the unique element of $\Gal\big(\bb{Q}(\mu_{p^N})/\bb{Q}\big)$ of  order $2$.
Since $KL/\bb{Q}(\mu_{p^N})$ is an abelian $p$-extension, we regard $\Gal\big(KL/\bb{Q}(\mu_{p^N})\big)$ as a $\bb{Z}_p[\Gal(\bb{Q}(\mu_{p^N})/\bb{Q})]$-module. We have a decomposition 
$$\Gal\big(KL/\bb{Q}(\mu_{p^N})\big)=\Gal\big(KL/\bb{Q}(\mu_{p^N})\big)_{+}\times\Gal\big(KL/\bb{Q}(\mu_{p^N})\big)_{-} \ ,$$  
where $\Gal\big(KL/\bb{Q}(\mu_{p^N})\big)_{+}$ \ (resp.\ $\Gal\big(KL/\bb{Q}(\mu_{p^N})\big)_{-}$) \ denotes the subgroup of \\
$\Gal\big(KL/\bb{Q}(\mu_{p^N})\big)$ on which $c$ acts trivially (resp.\ by $-1$).

The element $c$ acts on $\Gal\big(K/\bb{Q}(\mu_{p^N})\big)$ trivially since $K$ is the extension field of $\bb{Q}(\mu_{p^N})$ generated by all elements of $F_m\{n\}_\chi F_m(n)$, which is totally real. On the other hand, The element $c$ acts on $\Gal\big(L/\bb{Q}(\mu_{p^N})\big)$ by $-1$ since $\Gal\big(\bb{Q}(\mu_{p})/\bb{Q}\big)$ acts on $\Gal\big(L/\bb{Q}(\mu_{p^N})\big)$ via the character $\chi^{-1}\omega$, where $\omega \in \Hom\big(\Gal\big(\bb{Q}(\mu_{p})/\bb{Q}\big), \mu_{p-1}\big)$ is the Teichm\"uler character. Hence we have $$K \cap L= \bb{Q}(\mu_{p^N}).$$

{\em The fourth step. We shall complete the proof}.
 
By the third step and the condition (A), there exists infinitely many prime numbers $q'$ such that 
$$\begin{cases}
(q'_K,K/\bb{Q})=\sigma \\
(q'_L,L/\bb{Q})=\lambda'. 
\end{cases}$$ 
We shall prove that each of such $q'$ unramified in $L/\bb{Q}$ satisfies conditions (1)-(3) of Proposition \ref{chebo appli}.

First, we show $q'$ satisfies conditions (1) and (2). Let $\alpha=(\alpha_v)_v \in \bb{A}_{F_m}^\times$ be an idele whose $q'_{F_m}$-component is a prime element of $F_{m,q'_{F_m}}$, and other components are $1$. 
Let $\beta=(\beta_v)_v \in \bb{A}_{F_m}^\times$ be an idele such that the $q_{F_m}$-component is a prime element of $F_{m,q_{F_m}}$, the $\prod_{v|n}F_{m,v}^\times$-component is $\tilde\tau_{\chi}^{-1}$ which is a lift of $\tau_{\chi}^{-1} \in \big(\prod_{v|n}k(v)^\times \otimes \bb{Z}_p \big)_{\chi}$ in $\prod_{v|n}\mca{O}_{F_{m,v}}^\times$, and other components are $1$. 
By definition, ideles $\alpha$ and $\beta$ have the same image in $$\bigg( \frac{(\prod_{v|n}F_{m,v}^\times/\mca{O}_{F_{m,v}}^1 )\times (\bigoplus_{u \nmid n}F_{m,u}^\times/\mca{O}_{F_{m,u}}^\times)}{ F_m^\times}\otimes\bb{Z}_p \bigg)_{\chi} \simeq  \Gal\big(F_m\{n\}_\chi/F_m\big).$$
This implies there exist $z \in (F_m^\times\otimes\bb{Z}_p)_\chi$ such that 
$$\alpha=z\beta \ \ \text{in} \  \bigg( \big((\prod_{v|n}F_{m,v}^\times/\mca{O}_{F_{m,v}}^1 )\times (\bigoplus_{u\nmid n}F_{m,u}^\times/\mca{O}_{F_{m,u}}^\times)\big)\otimes\bb{Z}_p\bigg)_\chi. $$
Hence, we have $(z)_{F_{m,\chi}}=(q'_{F_m}-q_{F_m})_\chi,$
and $\phi_{F_m,N,\chi}^{\ell_i}(z)=\tau_i$ for each $i=1,\dots,r$. The prime number $q'$ satisfies conditions (1) and (2).

Next, we shall prove $q'$ satisfies condition (3).
Since $q'$ is unramified in $L/\bb{Q}$, the group $W$ is contained in the kernel of $[\cdot ]_{F_m,N,\chi}^{q'}$.
Since $(q'_L,L/\bb{Q})=\lambda' $, for any $x \in W$, we have
$$\rho_{N-1}^{\mrm{pr} \circ \lambda(x)}=\lambda'(x)
=(x^{1/{p^N}})^{\Frob_{q'}-1},$$
where $\Frob_{q'} \in \Gal(L/\bb{Q})$ is the arithmetic Frobenius at $q'$, and $x^{1/{p^N}}\in L$ is a $p^N$-th root of $x$. 
Then, we obtain $$\rho_{N-1}^{\mrm{pr} \circ \lambda(x)} \equiv x^{(q'-1)/p^N} \pmod{\overline{{q'}_M}}.$$

There is  the (unique) intermediate field $M$ of $F_m(q')/F_m$ whose degree over $F_m$ is $p^N$ since $q'\equiv 1 \pmod{p^N}$. Let $\pi$ be the prime element of $M_{N,{q'}_M}$. By definition of $\sigma_{q'}$, we have 
$$\pi^{\sigma_{q'}-1} \equiv \rho_{N-1} \pmod{\overline{{q'}_M}},$$ where $\overline{{q'}_M}$ is the maximal ideal of $M_{{q'}_M}$.
Recall that $W$ is contained in the kernel of $[\cdot ]_{F_m,N,\chi}^{q'}$. By definition of $\bar{\phi}_{F_m,N,\chi}^{q'}$, we have 
$$ \pi^{\phi (x) -1} \equiv x^{(q'-1)/p^N} \pmod{\overline{{q'}_M}}$$
for all $x \in W$, where we put
$$\phi (x):= \sigma_{q'}^{\mrm{pr}\circ \bar{\phi}_{F_m,N,\chi}^{q'}(x)}.$$
Then, we have $$x^{(q'-1)/p^N} \equiv \rho_{N-1}^{\mrm{pr}\circ \bar{\phi}_{F_m,N,\chi}^{q'}(x)} \pmod{\overline{{q'}_M}}$$
for all $x \in W$.
Hence, we obtain $$\rho_{N-1}^{\mrm{pr} \circ \lambda(x)}= \rho_{N-1}^{\mrm{pr}\circ \bar{\phi}_{F_m,N,\chi}^{q'}(x)}$$
for all $x \in W$.
By Lemma \ref{onajimono}, we have $\lambda=\bar{\phi}_{F_m,N,\chi}^{q'}|_W$. 
Therefore $q'$ satisfies condition (3) of Proposition \ref{chebo appli}, and the proof is complete.
\end{proof}

\section{The cyclotomic ideals}\label{the section of cyclotomic ideals}
Here, we will define the $i$-th cyclotomic ideal $\mf{C}_i$  of $\Lambda$ for each $i \ge 0$ (cf.\ Definition \ref{The $i$-th cyclotomic ideal}). Recall we denote $\bb{Z}/p^N[\Gal(F_m/\bb{Q})]$ by $R_{F_m,N}$. 
First, we fix $m$ and $N$.

\begin{dfn}
For $n\in \mca{N}_N$, we define $W_{F_m,N}^n$ to be the $R_{F_m, N}$-submodule of $F_m^{\times}/p^N$ generated by $\{ \kappa_{m,N}(\xi) \ | \ \xi \in \mu_n  \}$.
\end{dfn}
\begin{dfn}
Recall we put $\epsilon(n):=r$ for $n=\prod_{i=1}^r \ell_i\in\mca{N}_N \ (\ell_i\in\mca{S}_N \ \text{for} \ i=1,\dots,r)$.  
We denote by $\mf{C}_{i,F_m,N}$ the ideal of $R_{F_m, N}$ generated by images of all $R_{F_m,N}$-homomorphisms 
$\xymatrix{f\colon W_{F_m,N}^n \ar[r] & R_{F_m,N}}$, 
where $n$ runs through all elements of $\mca{N}_N$ satisfying $\epsilon(n)\le i$.
\end{dfn}

\begin{rem}
Since we have decomposition $$\Hom_{R_{F_m,N}}(W_{F_m,N}^n, R_{F_m,N})=\bigoplus_{\chi \in \widehat{\Delta}} \Hom_{R_{F_m,N,\chi}}(W_{F_m,N,\chi}^n, R_{F_m,N,\chi}),$$
the $\chi$-part $\mf{C}_{i,F_m,N,\chi}$ of $\mf{C}_{i,F_m,N}$ coincides with the ideal generated by 
$\Im f_\chi$ of $R_{F_m, N,\chi}$, 
where $f_\chi$ runs through all elements of $\Hom_{R_{F_m,N,\chi}}(W_{F_m,N,\chi}^n, R_{F_m,N,\chi})$ for any $n \in\mca{N}_N$ with  $\epsilon(n)\le i$.
\end{rem}

Then, we will define the $i$-th cyclotomic ideal  $\mf{C}_i$ by taking the projective limit of $\{ \mf{C}_{i,F_m,N}\}_{m,N}$. To define it, we need the following lemma.

\begin{lem}
Let $m_1, m_2, N_1, N_2$ be integers satisfying $N_1\ge m_1+1$, $N_2\ge m_2+1$, $m_2 \ge m_1$, and $N_2 \ge N_1$. The image of  $\mf{C}_{i,F_{m_2},N_2}$ in $R_{F_{m_1}, N_1}$ by the natural surjection is contained in $\mf{C}_{i,F_{m_1},N_1}$. 
\end{lem}

\begin{proof}
It is sufficient to show the lemma in the following two cases: 
\begin{enumerate}
\item $(m_2,N_2)=(m_1, N_1+1)$;
\item $(m_2,N_2)=(m_1+1, N_1)$.
\end{enumerate}

In the case (1), our lemma is clear. We shall show the lemma in the case (2).

We put $m=m_1$, $N=N_1=N_2$, $R_1= R_{F_{m}, N}$, $R_2= R_{F_{m+1}, N}$, and the natural surjection $\xymatrix{\mrm{pr}\colon R_2 \ar[r] & R_1}$.
We show the following claim:
\begin{claim}
For $f_2 \in \Hom_{R_2}(W_{F_{m+1},N}^n, R_2)$, there exists a homomorphism $f_1 \in \Hom_{R_2}(W_{F_{m+1},N}^n, R_1)$ such that $\Im f_1 = \Im (\mrm{pr} \circ f_2).$
\end{claim}

For each elements $\sigma \in \Gal (F_m/\bb{Q})$, we fix a lift $\bar{\sigma} \in \Gal (F_{m+1}/\bb{Q})$ of $\sigma$. We have 
$$R_2^{\Gal(F_{m+1}/F_m)}=\{\sum_{\sigma \in \Gal (F_m/\bb{Q})}a_\sigma \bar{\sigma} \mb{n} \ | \ a_\sigma \in \bb{Z}/p^N \},$$
where $\mb{n}$ is an element of $R_2$ defined by 
$$\mb{n}:=\sum_{\tau \in \Gal (F_{m+1}/F_m)}\tau.$$
We define the isomorphism $\xymatrix{\varphi \colon  R_2^{\Gal(F_{m+1}/F_m)} \ar[r] & R_1}$ of $R_1$-modules by 
$$\sum_{\sigma \in \Gal (F_m/\bb{Q})}a_\sigma \bar{\sigma} \mb{n} \longmapsto \sum_{\sigma \in \Gal (F_m/\bb{Q})}a_\sigma \sigma.$$ Let $\xymatrix{\iota\colon F_{m}^{\times}/p^N \ar[r] & F_{m+1}^{\times}/p^N}$ be the canonical homomorphism. Since $F_m$ is totally real, the homomorphism $\iota$ is injective.  We have $$\mrm{pr}\circ f_2 = \varphi \circ f_2 \circ \iota \circ N_{F_{m+1}/F_m},$$
where $\xymatrix{N_{F_{m+1}/F_m}\colon F_{m+1}^{\times}/p^N \ar[r] &  F_{m+1}^{\times}/p^N}$ is induced by the norm map.
Note that it follows from Lemma \ref{norm} that we have $$N_{F_{m+1}/F_m}(W_{F_{m+1},N}^1)=W_{F_m,N}^1.$$
If we put $f_1 = \varphi \circ f_2 \circ \iota$, then we have $$\Im f_1 = \Im (\mrm{pr} \circ f_2).$$
Thus we have proved the claim.
Our lemma follows from the claim immediately.
\end{proof}

\begin{dfn}[The $i$-th cyclotomic ideal]\label{The $i$-th cyclotomic ideal}
We define {\em the $i$-th cyclotomic ideal}  $\mf{C}_i$ to be the ideal of $\Lambda$ such that $\mf{C}_i:=\plim\mf{C}_{i,F_m,N},$
where the projective limit is taken with respect to the system of the natural homomorphisms 
$\xymatrix{
\mf{C}_{i,F_{m_2},N_2} \ar[r] & \mf{C}_{i,F_{m_1},N_1} \\
}$
for integers $m_1, m_2, N_1, N_2$ satisfying $N_1\ge m_1+1$, $N_2\ge m_2+1$, $m_2 \ge m_1$ and $N_2 \ge N_1$.
\end{dfn}

We shall prove Proposition \ref{size}, which is a proposition on the size of $\mf{C}_0$. 
\begin{prop}
\label{size}
Let $\chi$ be a non-trivial character in $\widehat{\Delta}$. Then, 
$$\ann_{\Lambda_{\chi}}(X_{\mrm{fin},\chi})\cha_{\Lambda_{\chi}}\big( (\overline{\mca{O}_{\infty}^1}/\overline{C_{\infty}^1})_{\chi} \big) \subseteq \mf{C}_{0,\chi} \subseteq \cha_{\Lambda_{\chi}}\big( (\overline{\mca{O}_{\infty}^1}/\overline{C_{\infty}^1})_{\chi}\big).$$
\end{prop}

\begin{proof}
First, we will prove that $\mf{C}_{0,\chi}$ contains $\ann_{\Lambda_{\chi}}(X_{\mrm{fin},\chi})\cha_{\Lambda_{\chi}}\big( (\overline{\mca{O}_{\infty}^1}/\overline{C_{\infty}^1})_{\chi} \big)$ for a non-trivial character $\chi \in \widehat\Delta$. 
 By Proposition \ref{modules} (2), we can take an isomorphism $\xymatrix{\varphi\colon \overline{\mca{O}_{\infty}^1} \ar[r]^(0.58){\simeq} & \Lambda}$, and by Corollary \ref{coinv}, it induces an isomorphism 
$$\xymatrix{\bar{\varphi}_{F_m,N,\chi}\colon \big(N_{\infty}(\mca{O}_{F_m}^\times)/p^N\big)_\chi = \big(N_{\infty}(\mca{O}_{F_m}^1)/p^N\big)_\chi \ar[r]^(0.78){\simeq} & R_{F_m, N,\chi}. }$$  
It follows from Corollary \ref{gene.cyc.lim} that the image of $C_{F_m}^1$ in $F_m^{\times}/p^N$ coincides with $W_{F_m,N}^1$. 
It is sufficient to show that for any $\delta \in \ann_{\Lambda_\chi}(X_{\mrm{fin},\chi})$,
$$\delta \bar{\varphi}_{F_m,N,\chi} \big( \text{the image of } (\overline{C_{F_m}^1})_\chi \big) \subseteq  \psi(W_{F_m,N}^1)$$
for a homomorphism $\psi\in \Hom_{R_{F_m,N,\chi}}(W_{F_m,N,\chi}^1, R_{F_m,N,\chi})$ since this means that the image of $\mf{C}_{0,\chi}$ in $R_{F_m,N,\chi}$ contains the image of $\ann_{\Lambda_{\chi}}(X_{\mrm{fin},\chi})\cha_{\Lambda_{\chi}}\big( (\overline{\mca{O}_{\infty}^1}/\overline{C_{\infty}^1})_{\chi}\big)$ in $R_{F_m,N,\chi}$. We show a stronger assertion. 
\begin{claim}
Let $\mca{NO}_{F_m,N,\chi}$ be the image of the natural homomorphism $$\xymatrix{\big(N_{\infty}(\mca{O}_{F_m}^\times)/p^N\big)_\chi \ar[r] & \big(\mathcal{O}_{F_m}^\times/p^N\big)_\chi \subset \big(F_m^{\times}/p^N\big)_\chi.}$$
There exists a homomorphism  $\xymatrix{\psi\colon \mca{NO}_{F,N} \ar[r] & R_{F_m,N}}$ which makes the diagram
$$\xymatrix{
 \big(C_{F_m}^1/p^N\big)_\chi \ar[r] \ar@{->>}[d] & \big(N_{\infty}(\mca{O}_{F_m}^1)/p^N\big)_\chi \ar[rr]^(0.6){\delta \bar{\varphi}_{F_m,N,\chi}}  \ar[d] & & R_{F_m,N,\chi} \\
W_{F_m,N,\chi}^1  \ar@{^{(}->}[r] & (\mca{NO}_{F_m,N})_\chi \ar@{-->}[rru]_{\psi}  & \\
}
$$ 
commute.
\end{claim} 
This claim immediately follows from the following lemma.
\begin{lem}
\label{delta kowaza}
Let $\delta$ be an element of $\ann_{\Lambda_\chi}(X_{\mrm{fin},\chi})$. For any homomorphism 
$$\xymatrix{f\colon \big(N_{\infty}(\mca{O}_{F_m}^\times)/p^N\big)_\chi  \ar[r] & R_{F_m, N,\chi}}$$
of $R_{F_m, N,\chi}$-modules, there exists a homomorphism 
$\xymatrix{g\colon \mca{NO}_{F,N,\chi} \ar[r] & R_{F_m,N,\chi}}$ which makes the diagram 
$$\xymatrix{
\big(N_{\infty}(\mca{O}_{F_m}^1)/p^N)_\chi \ar[r]^(0.6){\delta f}  \ar[d] &  R_{F_m,N,\chi} \\
\mca{NO}_{F_m,N,\chi} \ar@{-->}[ru]_{g}  & \\
}
$$ 
commute.
\end{lem} 

\begin{proof}[Proof of Lemma \ref{delta kowaza}]
We consider the following commutative diagram
$$\xymatrix{
0 \ar[r] & \big( N_{\infty}(\mca{O}_{F_m}^\times)\otimes\bb{Z}_p \big)_\chi \ar[r] \ar[d]^{\times p^N} & \big(\mca{O}_{F_m}^\times\otimes\bb{Z}_p \big)_\chi \ar[r] \ar[d]^{\times p^N} & \big( (\mca{O}_{F_m}^\times/ N_{\infty}(\mca{O}_{F_m}^\times) )\otimes\bb{Z}_p  \big)_\chi  \ar[r] \ar[d]^{\times p^N} & 0 \\
0 \ar[r] & \big( N_{\infty}(\mca{O}_{F_m}^\times)\otimes\bb{Z}_p \big)_\chi \ar[r]  & \big(\mca{O}_{F_m}^\times\otimes\bb{Z}_p \big)_\chi \ar[r]  & \big( (\mca{O}_{F_m}^\times/ N_{\infty}(\mca{O}_{F_m}^\times))\otimes\bb{Z}_p  \big)_\chi \ar[r]  & 0, \\
}
$$
where two rows are exact, and all vertical arrows $\times p^N$ are the homomorphisms taking $p^N$-th power. 
Applying the snake lemma, we find that the kernel of the natural homomorphism 
$\xymatrix{\big(N_{\infty}(\mca{O}_{F_m}^\times)/p^N\big)_\chi \ar[r] & (\mca{O}_{F_m}^\times/p^N)_\chi}$ 
is a subquotient of the module $$\big( (\mca{O}_{F_m}^\times/ N_{\infty}(\mca{O}_{F_m}^\times))\otimes\bb{Z}_p  \big)_\chi\simeq \big(\overline{\mca{O}_{F_m}^1}/N_{\infty}(\overline{\mca{O}_{F_m}^1}) \big)_\chi,$$ 
which is annihilated by $\delta$ by Corollary \ref{delta}. This implies Lemma \ref{delta kowaza}.
\end{proof}

Next, we will prove that $\mf{C}_{0,\chi}$ is contained in $\cha_{\Lambda_{\chi}}\big( (\overline{\mca{O}_{\infty}^1}/\overline{C_{\infty}^1})_{\chi}\big)$ for a non-trivial character $\chi \in \widehat\Delta$. 
Note that $R_{F_m,N,\chi}$ is  an injective $R_{F_m,N,\chi}$-module since the $R_{F_m,N,\chi}$-module $\Hom_{\bb{Z}}(R_{F_m,N,\chi},\bb{Q}/\bb{Z})$ is injective and free of rank $1$.
Let $\xymatrix{f\colon W_{F_m,N,\chi}^1 \ar[r] & R_{F_m,N,\chi}}$ be an $R_{F_m,N, \chi}$- homomorphism.
Since $R_{F_m,N,\chi}$ is an injective $R_{F_m,N,\chi}$-module, there exists a homomorphism $\xymatrix{\tilde{f} \colon \big( \mca{O}^{\times}_{F_m}/p^N \big)_{\chi} \ar[r] & R_{F_m,N,\chi}}$
whose restriction to $W_{F_m,N,\chi}^1$ coincides with $f$. 
Then, we have an element $a \in R_{F,N,\chi}$  which makes the following diagram
$$\xymatrix{
 \big( C_{F_m}^1/p^N \big)_{\chi} \ar[r] \ar@{->>}[rdd]  & \big( N_{\infty}(\mca{O}^1_{F_m})/p^N \big)_{\chi} \ar[rr]_(0.62){\simeq }^(0.6){\bar{\varphi}_{F_m,N,\chi}} \ar[d] &  & R_{F,N,\chi} \ar@{.>}[d]^{\times a}  \\
 & \big( \mca{O}^{\times}_{F_m}/p^N \big)_{\chi} \ar@{-->}[rr]^(0.55){\tilde{f}} &  & R_{F,N,\chi}   \\
 & W_{F_m,N,\chi}^1 \ar@{^{(}->}[u] \ar[urr]_{f} &  &   \\
}$$ 
commute, where $\times a$ is the homomorphism multiplying $a$. This diagram implies that $f(W_{F_m,N,\chi}^1)$ is contained in the image of $\cha_{\Lambda_{\chi}}\big( (\overline{\mca{O}_{\infty}^1}/\overline{C_{\infty}^1})_{\chi} \big)$ in $R_{F_m,N,\chi}$.
Therefore, we obtain $\mf{C}_{0,\chi} \subseteq \cha_{\Lambda_{\chi}}\big( (\overline{\mca{O}_{\infty}^1}/\overline{C_{\infty}^1})_{\chi} \big).$
\end{proof}

Theorem \ref{Main theorem} for $i=0$ follows from Proposition \ref{size} and the Iwasawa main conjecture.
\begin{cor}[Theorem \ref{Main theorem} for $i=0$]\label{i=0}
Let $\chi $ be a non-trivial character in $\widehat \Delta$. 
\begin{enumerate}
\item $\mf{C}_{0,\chi} \subseteq \Fitt_{\Lambda_\chi, 0}(X'_{\chi})$.
\item $\ann_{\Lambda_\chi}(X_{\mrm{fin}, \chi}) \Fitt_{\Lambda_\chi, 0}(X'_{\chi})
\subseteq \mf{C}_{0,\chi}$. 
\end{enumerate}
\end{cor}

Note that 
Corollary \ref{i=0} is a restatement of the Iwasawa main conjecture in terms of 0-th cyclotomic ideal. Indeed, we can obtain $\cha_{\Lambda_\chi}(X_{\chi})$ from this Corollary.

\begin{rem}\label{0-layer}
Here, we remark on the structure of $A_0$. 
Using the Kolyvagin derivatives and the Euler system arguments, Rubin determined the structure of $A_0$ completely. 
(See \cite{Ru2} for the minus-part version of this result.) 
We can show that our $\mf{C}_{i,F_0,N,\chi}$ is equal to the ideal of $R_{F_0,N,\chi}$ generated by $p^{\partial(i,N,\chi)}$, 
where $\partial(i,N,\chi)$ is the largest integer satisfying 
$$\kappa_{0,N}(\xi_n)_\chi \in \bigg( \big( F_0^\times/(F_0^\times)^{p^N} \big)_\chi\bigg)^{p^{\partial(i,N,\chi)}}$$
for any $n$ with $\epsilon(n) \le i$.    
Hence, when $\chi \in \widehat\Delta$ is non-trivial and $N$ is sufficiently large, 
Rubin's result can be described as 
$$A_{0,\chi} \simeq \bigoplus_{i \ge 0}R_{F_0,N,\chi}/p^{\partial(i,N,\chi)-\partial(i+1,N,\chi)} \simeq \bigoplus_{i \ge 0}\mf{C}_{i+1,F_0,N,\chi}/\mf{C}_{i,F_0,N,\chi}$$ in our notation. 
Equivalently, we have $\Fitt_{R_{F_0,N,\chi},i}(A_0)= \mf{C}_{i,F_0,N,\chi}$ for all $i \ge 0$, non-trivial $\chi \in \widehat\Delta$ and sufficiently large $N$.
We also remark that Mazur and Rubin proved a general theorem in \cite{MR} Theorem 4.5.9, which contains the above result as a special case. 
\end{rem}

\section{Proof of the main theorem}

In this section, we prove Theorem \ref{Main theorem}. Our argument is almost parallel to \cite{Ku} \S 9, but we have to treat the pseudo-null-part $X_{\rm{fin}}$ of $X$ and the unit group $\mca{O}_{F_m}^\times$ carefully.

First, we recall the notation and the statement of the theorem. The $\Lambda$-module $X$ is defined by $X:= \plim A_m$, where $A_m$ is the $p$-Sylow subgroup of the ideal class group of $F_m$ and the projective limit is taken with respect to the norm map.
The $\Lambda$-module $X'$ is defined by  $X':=X/X_{\mrm{fin}}$, where $X_{\mrm{fin}}$ is the maximal pseudo-null $\Lambda$-submodule of $X$. Our main theorem is as follows.
\begin{thm}[Theorem \ref{Main theorem}]\label{main theorem}
Let $\chi $ be a non-trivial character in $\widehat \Delta$. 
\begin{enumerate}
\item $\mf{C}_{0,\chi} \subseteq \Fitt_{\Lambda_\chi, 0}(X'_{\chi})$.
\item $\ann_{\Lambda_\chi}(X_{\mrm{fin}, \chi})\Fitt_{\Lambda_\chi, i}(X'_{\chi})\subseteq \mf{C}_{i, \chi}$ for  $i \ge 0$.
\end{enumerate} 
\end{thm}
We have already proved Theorem \ref{Main theorem} for $i=0$ in the last section. Here, we prove the second assertion for $i \ge 1$.  

\subsection{}
We spend this subsection on the setting of notations.

We assume that $\chi \in \widehat{\Delta}$ is non-trivial. 
Since $X'_{\chi}$ has no non-trivial pseudo-null submodules, we have an exact sequence
\begin{align}\label{su}
\xymatrix{0 \ar[r] & \Lambda_\chi^h \ar[r]^{f} & \Lambda_\chi^h \ar[r]^{g} & X'_{\chi} \ar[r] & 0,} 
\end{align}
by Lemma \ref{no pn}.
Let $M$ be the matrix corresponding to  $f$ with respect to the standard basis $( \mb{e}_i )_{i=1}^h$ of $\Lambda_\chi^h$. Let $\{ m_1,...,m_h \}$ and $\{ n_1,...,n_h\}$ be permutations of $\{1,...,h\}$.
For any integer $i$ satisfying $1 \le i \le h-1$, consider the matrix $M_i$ which is obtained from $M$ by eliminating the $n_j$-th rows ($j=1,...,i$) and  the $m_k$-th columns ($k=1,...,i$). Here, we shall prove that $\delta \det M_i \in \mf{C}_{i,\chi}$ for any $i \in \bb{Z}_{\ge 1}$  for any  $\delta \in \ann (A_\mrm{fin})$. If $\det M_i=0$, this is trivial, so we assume that $\det M_i \not= 0$. If necessary, we permute $\{ m_1,...,m_i \}$, and assume $\det M_r \not= 0$ for all integers $r$ satisfying $ 0 \le r \le i$.

For each $m \in \bb{Z}_{\ge 0} $, we take a positive integer $N_m$ such that and $N_{m+1}> N_m > m+1$, and $p^{N_m} > \# A_m$  for any $m \in \bb{Z}_{\ge 0} $. 
For simplicity, we put $F:=F_m$, $R:=\bb{Z}_p[\Gal(F_m/F_0)]_\chi$,  $N:=N_m$ and $R_N:=R_{F,N,\chi}=\bb{Z}/p^N[\Gal(F_m/F_0)]_\chi$. 
Let $A_{m,\mrm{fin},\chi}$ be the image of $X_{\mrm{fin},\chi}$ in $A_{m,\chi}$ by the natural homomorphism. 
By Proposition \ref{ideal class}, the $R$-module $A'_{m,\chi}:=X'_{\Gamma_m,\chi}$ is regarded as the quotient $A_{m,\chi}/A_{m,\mrm{fin},\chi}$. 
From the exact sequence (\ref{su}), we obtain the exact sequence 
$$\xymatrix{0 \ar[r] & R^h \ar[r]^{\bar{f}} & R^h \ar[r]^{\bar{g}} & A'_{m,\chi} \ar[r] & 0,}$$
by taking the $\Gamma_m$-coinvariants.

The image of $\mb{e}_r$ in $R^h$ is denoted by $\mb{e}_i^{(m)}$. We define $\mb{c}_1:=g(\mb{e}_1),\dots, \mb{c}_h:=g(\mb{e}_h)$, and  $\mb{c}_r^{(m)}$ to be the image of $\mb{c}_r$ in $A_{m,\chi}/A_{m,\mrm{fin},\chi}$, namely $\mb{c}_r^{(m)}:= \bar{g}(\mb{e}_r^{(m)})$. We fix a lift $\tilde {\mb{c}}_r^{(m)} \in A_{m,\chi}$ of $\mb{c}_r^{(m)}$, and define $$P_r:=\{ \ell \in \mca{S}_N \ | \ [\ell_F]_\chi = \tilde {\mb{c}}_r^{(m)} \},$$
where $[\ell_F]_\chi$ is the class of $\ell_F$ in $A_{m,\chi}$.  
We define $P:=\bigcup_{r=1}^iP_r$, and $P_F$ to be the set of all the prime ideals of $F$ above $P$. Let $J$ be the subgroup of $\mca{I}_F$ generated by $P_F$, and the $R$-submodule $\mca{F}$  of $(F^\times\otimes\bb{Z}_p)_\chi$ the inverse image of $(J\otimes\bb{Z}_p)_{\chi}$ by the homomorphism  $\xymatrix{(\cdot )_F\colon ( F^{\times} \otimes \bb{Z}_p)_\chi \ar[r] &  (\mca{I}_F\otimes \bb{Z}_p)_{\chi}}$. We define a surjective homomorphism $$\xymatrix{\alpha\colon (J\otimes \bb{Z}_p)_\chi \ar[r] & R^h}$$ by $\ell_F \mapsto \mb{e}_r$ for each $\ell \in P_r $ and $r $ with $1 \le r \le h$. We define $$\xymatrix{\alpha_r := \mrm{pr}_r \circ \alpha\colon (J\otimes \bb{Z}_p)_\chi \ar[r]^(0.76){\alpha} & R^h  \ar[r]^{\mrm{pr}_r} & R}$$
to be the composition of $\alpha$ and the $r$-th projection.

We define the homomorphism $\xymatrix{\beta\colon \mca{F} \ar[r] & R^h}$ to make the following diagram  
\begin{align}\label{diagram2}
\xymatrix{
& \mca{F} \ar[r]^(0.38){(\cdot)_{F, \chi}} \ar[d]^\beta & (J\otimes \bb{Z}_p)_\chi \ar@{->>}[r]^(0.6){\mrm{can.}} \ar[d]^\alpha &  A'_{m,\chi} \ar@{=}[d] &  \\
0 \ar[r] & R^h \ar[r]^{\bar{f}} & R^h \ar[r]^{\bar{g}} & A'_{m,\chi} \ar[r] & 0 \\
}
\end{align}
commute, where $\mrm{can.}$ is induced by the canonical homomorphism $J \longrightarrow A'_{m,\chi}= A_{m,\chi}/A_{m,\mrm{fin},\chi}$. Note that since the second row of the diagram is exact, $\beta$ is well-defined. We define $$\xymatrix{\beta_r := \mrm{pr}_r \circ \beta\colon \mca{F} \ar[r]^(0.71){\beta} & R^h \ar[r]^{\mrm{pr}_r} & R}$$
to be the composition of $\beta$ and the $r$-th projection.

We consider the diagram (\ref{diagram2}) by taking $( -  \otimes \bb{Z}/p\bb{Z})$. 
First, we prove the following two lemmas, namely Lemma \ref{tannsha} and \ref{delta kowaza 2}.
\begin{lem}\label{tannsha}
The canonical homomorphism $\xymatrix{\mca{F}/p^N \ar[r] & (F^{\times}/p^N)_{\chi}}$ is injective.
\end{lem}

\begin{proof}
Let $x$ be an element in the kernel of the homomorphism $\xymatrix{\mca{F}/p^N \ar[r] & (F^{\times}/p^N)_{\chi}}$ and $\tilde{x}$ a lift of $x$ in $\mca{F}$. Then, there exists  $y \in (F^{\times} \otimes \bb{Z}_p)_\chi$ such that $\tilde{x}=y^{p^N}$. 
Since $(\tilde{x})_{F, \chi} \in (J\otimes \bb{Z}_p)_\chi$ and $(\mca{I}_F\otimes \bb{Z}_p)/(J\otimes \bb{Z}_p)$ is torsion free $\bb{Z}_p$-module, we have $ (y)_{F, \chi} \in (J\otimes \bb{Z}_p)_\chi$.
Hence, $y \in \mca{F}$, and we obtain $x=1$.
\end{proof}
The $R_N$-module $\mca{F}/p^N$ is regarded as a submodule of $(F^{\times}/p^N)_{\chi}$ by Lemma \ref{tannsha}. 

We regard $(F^{\times}/p^N)_\chi$ as a $\Lambda_\chi$-module. For an element $x \in (F^{\times}/p^N)_{\chi}$ and $\delta \in \ann_{\Lambda_\chi}(X_{\mrm{fin},\chi})$, we denote the scaler multiple of $x$ by $\delta \in \Lambda_\chi$ by $x^\delta$.

\begin{lem}
\label{delta kowaza 2}
Let $[\cdot ]_{F,N,\chi}$ be the homomorphism $\xymatrix{(F^{\times}/p^N)_\chi \ar[r] & (\mca{I}_F/p^N)_{\chi}}$ induced by $\xymatrix{(\cdot )_F\colon F^\times \ar[r] & \mca{I}_F}$. Let $x$ be an element of $(F^{\times}/p^N)_\chi$ such that 
$[x]_{F,N,\chi} \in (J/p^N)_\chi.$
Then, for any $\delta \in \ann_{\Lambda_\chi}(X_{\mrm{fin},\chi})$, $x^\delta$ is contained in $\mca{F}/p^N \subset (F^{\times}/p^N)_\chi$.
\end{lem}

\begin{proof}
Recall the canonical exact sequence:
$$\xymatrix{0 \ar[r] & \mca{P} \ar[r] & \mca{I}_K \ar[r] & A_m \ar[r] & 0,} $$
where $\mca{P}$ is defined by $\mca{P} = F^\times/\mca{O}_{F_m}^\times$. By the snake lemma for the commutative diagram
$$\xymatrix{
0 \ar[r] & \mca{P} \ar[r] \ar[d]^{\times p^N} & \mca{I}_F \ar[r] \ar[d]^{\times p^N} & A_m \ar[r] \ar[d]^{\times p^N} & 0 \\
0 \ar[r] & \mca{P} \ar[r]  & \mca{I}_F \ar[r]  & A_m \ar[r]  & 0, \\
}
$$
we obtain the following exact sequence 
$$\xymatrix{0 \ar[r] & A_m \ar[r] & \mca{P}/p^N \ar[r] & \mca{I}_K/p^N \ar[r] & A_m \ar[r] & 0. }\ \  (\text{Recall}\ \ p^{N_m} > \# A_m.)$$

Let $B_m$ be the image of $J$ in $A_m$, and $\mca{P}_0 = \mca{F}/\mca{O}_{F_m}^\times$. Then, we have the exact sequence  
$$\xymatrix{0 \ar[r] & \mca{P}_0 \ar[r] & J \ar[r] & B_m \ar[r] & 0,} $$
and by a similar argument as above, we obtain the exact sequence
$$\xymatrix{0 \ar[r] & B_m \ar[r] & \mca{P}_0/p^N \ar[r] & J/p^N \ar[r] & B_m \ar[r] & 0.}$$ 
Now, we obtain the commutative diagram
\begin{align}\label{kernelfinite}
\xymatrix{
0 \ar[r] & B_m \ar[r] \ar@{^{(}->}[d] & \mca{P}_0/p^N \ar[r] \ar@{^{(}->}[d] & J/p^N \ar[r] \ar@{^{(}->}[d] & B_m \ar[r] \ar@{^{(}->}[d] & 0 \\
0 \ar[r] & A_m \ar[r] & \mca{P}/p^N \ar[r] & \mca{I}_F/p^N \ar[r]  & A_m \ar[r]  & 0 \\
}
\end{align}
whose two rows are exact, and the vertical arrows are injective.
Since $\delta A_m$ is contained in $B_m$, our lemma follows from a diagram chase. 
\end{proof}

From the first row of the diagram (\ref{kernelfinite}), we obtain the following corollary immediately.

\begin{cor}\label{yuugensei}
The order of the kernel of the homomorphism
$\xymatrix{[\cdot]_{F, \chi}\colon \mca{F}/p^N \ar[r] & J/p^N } $
is finite.
\end{cor}

Let $n$ be an element of $\mca{N}_N$ whose prime divisors are in $P$. We define $P^n_F$ to be the set of all elements of $P$ dividing $n$. We define $J_n$ to be the subgroup of $J$ generated by $P^n_F$,  and the submodule $\mca{F}_{n,N}$ of $\mca{F}/p^N$ the inverse image of $J_n$ by the restriction of $[\cdot ]_{F,N,\chi}$ to $\mca{F}/p^N$. Note that $\mca{F}_{n,N}$ is a {\em finite} $R_N$-submodule of $(F^{\times}/p^N)_\chi$ by Corollary \ref{yuugensei}. We have obtained the following commutative diagram

$$\xymatrix{
\mca{F}_{n,N} \ar[r]^(0.42){[\cdot]_{F, \chi}} \ar[d]^\beta & (J_n/p^N)_\chi \ar[d]^\alpha   \\
 R_N^h \ar[r]^{\bar{f}} & R_N^h \ .   \\
}$$

\subsection{}
Let $\delta$ be a non-zero element of $\ann_{\Lambda_{\chi}}(X_{\mrm{fin}})$.
In this and the next subsection, we write $\bar{\phi}^\ell$ in place of $\bar{\phi}_{F,N,\chi}^\ell$ for simplicity. 
Here, as in \cite{Ku}, we shall take the element $x_{\nu,q}  \in (F^\times/p^N)_\chi$ which is defined in Definition \ref{x}, to translate $\beta_r$ to homomorphisms of the type $\bar{\phi}^\ell$. Recall the element $x_{\nu,q}  \in (F^\times/p^N)_\chi$ is determined by $q$, $\nu$, and $\{ a_\ell \}_{\ell |\nu}$. We shall take them as follows.  

First, we take a prime number $q$ by the following way. For each integer $r$ with $1 \le r \le h$, we fix a prime number $q_r \in P_{n_r}$.
We put $Q:=\prod_{r=1}^h q_r \in \mca{N}_N$. By the Iwasawa main conjecture, we can take an isomorphism $$\xymatrix{\varphi\colon (\overline{\mca{O}_{\infty}^1})_{\chi} \ar[r]^(0.64){\simeq} & \Lambda_{\chi}}$$ which sends $\big(u\cdot\mrm{cyc}(\rho_m)_{\chi}\big)_{m \ge 0}$ to $\det M_0$. Let 
$$\xymatrix{\bar{\varphi}_{F,N, \chi}\colon \big(N_{\infty}(\mca{O}_{F}^1)/p^N\big)_{\chi} \ar[r]^(0.74){\simeq} & R_N}$$ be the induced isomorphism by $\varphi$. 
Recall that we define $\mca{NO}_{F_m,N,\chi}$ to be the image of the natural homomorphism $$\xymatrix{\big(N_{\infty}(\mca{O}_{F_m}^\times)/p^N\big)_\chi \ar[r] & \big(\mathcal{O}_{F_m}^\times/p^N\big)_\chi \subset \big(F_m^{\times}/p^N\big)_\chi.}$$
By Lemma \ref{delta kowaza}, there exists a homomorphism 
$\xymatrix{\psi\colon \mca{NO}_{F,N,\chi} \ar[r] & R_N} $ which makes the diagram
$$\xymatrix{
 (C_F^1/p^N)_\chi \ar[r] \ar@{->>}[d] & \big(N_{\infty}(\mca{O}_{F}^1)/p^N\big)_{\chi} \ar[rr]^(0.63){\delta \bar{\varphi}_{F,N, \chi}}  \ar[d] & & R_N \\
W_{F_m,N,\chi}^1  \ar@{^{(}->}[r] & \mca{NO}_{F,N,\chi} \ar@{-->}[rru]_{\psi}  & \\
}
$$ 
commute. By Proposition \ref{chebo appli}, we can take a prime number $q\in \mca{S}_N$ satisfying the following two conditions:
\begin{enumerate}
\item[(q1)] the class of $q_{F_m}$ in $A_{m,\chi}$ coincides with the class of ${q_1}_F$;
\item[(q2)] $\mca{NO}_{F,N,\chi}$ is contained in the kernel of $[\cdot ]_{F,N,\chi}^{q}$, and for all $x \in \mca{NO}_{F,N,\chi}$,$$\psi(x)=\bar{\phi}^{q}(x).$$ 
\end{enumerate}
In particular, we have 
\begin{align*}
\bar{\phi}^{q}(\mrm{cyc}(\rho_m)) &= \bar{\phi}^{q}(u \cdot \mrm{cyc}(\rho_m)) \\
&= \delta \bar{\varphi}_{F,N, \chi}(\mrm{cyc}(\rho_m)) \\
&= \delta \det M_0.
\end{align*}
We replace $q_1$ by $q$.

Next, we shall take  $\nu$ and $\{ a_\ell \}_{\ell |\nu}$. 

First, we consider the homomorphism $\xymatrix{\beta_{m_1}\colon \mca{F}_{Q,N} \ar[r] & R_N}$. Applying Proposition \ref{chebo appli}, we can take $\ell_2 \in \mca{S}_N(F(Q))$ such that $\ell_2 \in P_{n_2}$, $\ell\not=q_2$, and $$\beta_{m_1}(x)= \bar{\phi}^{\ell_2}(x) $$ for all $x \in \mca{F}_{Q,N}$. We put $\nu_1:=1$.

In the case $i=1$, we put $\nu:=\nu_1=1$, and $x_{\nu,q}=x_{1,q}= \kappa_{q}(\xi_q)$. It follows from Proposition \ref{[] and phi and x} (1) and Lemma \ref{delta kowaza 2} that $x_{1,q}^\delta $ is an element of $\mca{F}_{Q ,N}$.  

Suppose $i \ge 2$. To take $\nu$ and $\{ a_\ell \}_{\ell |\nu}$, we choose prime numbers $\ell_r$ for each $r$ with $2\le r \le i+1$  by induction on $r$ as follows.   
Let $r$ be an integer satisfying $2<r \le i+1$, and suppose that we have chosen distinct prime numbers $\ell_{s} \in \mca{S}_N(F(Q\nu_{s-1}))$ for each $s$ with $2 \le s \le r-1$. We put $\nu_{r-1} := 
\prod_{s=2}^{r-1}\ell_s$. 
We consider the homomorphism $\xymatrix{\beta_{m_1}\colon \mca{F}_{Q \nu_{r-1},N} \ar[r] & R_N.}$ 
Applying Proposition \ref{chebo appli}, we can take $\ell_r \in \mca{S}_N(F(Q\nu_{r-1}))$ satisfying the following conditions:
\begin{itemize}
\item[(x1)] $\ell_r \in P_{n_r}$, and $\ell \not=q_r$;
\item[(x2)] there exists $b_r \in (F^\times \otimes \bb{Z}_p)_\chi$ such that $(b_r)_{F, \chi} =(\ell_{r,F}-q_{r,F})_{\chi}$
 and  $\bar{\phi}^{\ell_s}(b_r)=0$ for any $s$ with $2 \le s < r$;  
\item[(x3)] $\bar{\phi}^{\ell_r}(x)= \beta_{m_{r-1}}(x)$ for any $x \in \mca{F}_{Q \nu_{r-1},N}$.
\end{itemize}
Thus, we have taken $\ell_2, \dots , \ell_{i+1}$, and we put $\nu:=\nu_{i}=\prod_{r=2}^i\ell_r\in \mca{N}_N$. For each $r$ with $2 \le r \le i$, we put $a_{\ell_r}:=-{\phi}^{\ell_r}(b_r) \in R_N \otimes H_{\ell_r},$
and we obtain $x_{\nu,q} \in (F^\times/p^N)_\chi$. It follows from Proposition \ref{[] and phi and x} (1) and Lemma \ref{delta kowaza 2} that $ x_{\nu,q}^\delta$ is an element of $\mca{F}_{Q \nu,N}$. Note that $q\nu$ is {\em well-ordered}. 

\subsection{}
In this subsection, we observe two homomorphism $\alpha$ and $\beta$ by using $x_{n,q}$, and describe $\det M_i$ in $R_N$. 
First, we prepare the following lemma.
\begin{lem}[cf. \cite{Ku} Lemma 9.2]
\label{ab}
Suppose $i \ge 2$. Then,
\begin{enumerate}
\item $ \beta_{m_{r-1}}(x_{\nu,q}^\delta )=0$ for all $r$ with $2 \le r \le i$;
\item $\alpha_j([x_{\nu,q}]_{F,\chi})=0$ for any $j\not=n_1,\dots,n_i.$ 
\end{enumerate}
\end{lem}

\begin{proof}
The second assertion (2) of the above lemma is clear by Proposition \ref{[] and phi and x} (1). 

We shall prove the first assertion. For any $r$ satisfying $2 \le r \le i$, we have $\alpha ([b_r]_{F,\chi})=0$ since $(b_r)_{F, \chi} =(\ell_{r,F}-q_{r,F})_{\chi}$.
Then, by the definition of $\beta$, we have $\beta(b_r)=0.$
We put $$y_r=x_{\nu,q}\prod_{s=r}^ib_s^{\bar{\phi}^{\ell_s}(x_{\nu/\ell_s,q})}, $$
then we have $\beta( x_{\nu,q}^\delta)=\beta( y_r^\delta)$. We will prove  $ \beta_{m_{r-1}}(y_r^\delta)=0$ for any $r$ satisfying $2 \le r \le i$. 

By Proposition \ref{[] and phi and x} (2), we have $ [y_r]_{F,N,\chi} \in J_{Q \nu_{r-1}}$, and then, by Lemma \ref{delta kowaza 2}, we have $y_r^\delta \in \mca{F}_{Q \nu_{r-1},N}$. Therefore, we obtain $$\delta \bar{\phi}^{\ell_r}( y_r)= \beta_{m_{r-1}}(y_r^\delta )$$  
by the condition (x3). Since $\bar{\phi}^{\ell_r}( b_{s})=0$ for all integers $s$ satisfying $r+1 \le s \le i$ by the condition (x2), we have 
$$\bar{\phi}^{\ell_r}( y_r) =
 \bar{\phi}^{\ell_r}( x_{\nu,q} b_r^{\bar{\phi}^{\ell_r}(x_{\nu/\ell_r,q})}).$$
By Proposition \ref{[] and phi and x} (3), we have $$\bar{\phi}^{\ell_r}( x_{\nu,q} b_r^{\bar{\phi}^{\ell_r}(x_{\nu/\ell_r,q})}) = - \bar{\phi}^{\ell_r}(b_r)\bar{\phi}^{\ell_r}(x_{\nu/\ell_r,q})+\bar{\phi}^{\ell_r}(x_{\nu/\ell_r,q})\bar {\phi}^{\ell_r}(b_r)=0.$$
Therefore, we obtain $ \beta_{m_{r-1}}( x_{\nu,q}^\delta)=\beta_{m_{r-1}}(y_r^\delta)=0. $
\end{proof} 

The goal of this subsection is the following proposition.

\begin{prop}[cf. \cite{Ku} p.44]
\label{ind}
We have the following equalities on elements of $R_N$:

\begin{enumerate}
\item $\delta (\det M) \bar{\phi}^{\ell_2}(x_{1,q})= \pm \delta^2 (\det M_1) \bar{\varphi}_{F,N, \chi}(\mrm{cyc}(\rho_m)_\chi)$; 
\item 
$\delta (\det M_{r-1}) \bar{\phi}^{\ell_{r+1}}(x_{\nu_{r},q}) = \pm \delta (\det M_{r}) \bar{\phi}^{\ell_{r}}(x_{\nu_{r-1},q})$ for any $r$ with $2 \le r \le i $,
\end{enumerate}
where the signs $\pm$ in (1) and (2) do not depend on $m$.
\end{prop}

\begin{proof}
For each $r$ satisfying $1 \le r \le i $ we put $$\mb{x}^{(r)}:=\beta(x_{\nu_r,q}^\delta) \in R_N^h\ \ \text{and}\ \ \mb{y}^{(r)}:=\alpha(x_{\nu_r,q}^\delta) \in R_N^h,$$
and regard them as column vectors.
Then, we have $\mb{y}^{(r)}=M\mb{x}^{(r)}$
in $R_N^h$.

We first prove the assertion (1) of the above proposition.
Since $x_{1,q}^{\delta}$ is an element of $\mca{F}_{q,N}$, we have 
\begin{align*}
\mb{y}^{(1)}&=\delta[\kappa_q(\xi_q)_\chi]_{F,N,\chi}^q\mb{e}_{n_1}^{(m)} & \\
&=-\delta \bar{\phi}^{q}(\mrm{cyc}(\rho_m)_\chi)\mb{e}_{n_1}^{(m)} &( \text{by Proposition \ref{[] and phi} (2)}) \\
&=-\delta^2 \bar{\varphi}_{F,N, \chi}(\mrm{cyc}(\rho_m)_\chi)\mb{e}_{n_1}^{(m)} &( \text{by condition (q2)}).
\end{align*}
Let $\widetilde{M}$ be the matrix of cofactors of $M$. Multiplying the both sides of $\mb{y}^{(1)}=M\mb{x}^{(1)}$ by $\widetilde{M}$,
and comparing the $m_1$-st components, we obtain 
$$(-1)^{n_1+m_1+1}(\det M_1) \delta^2 \bar{\varphi}_{F,N, \chi}(\mrm{cyc}(\rho_m)_\chi)=(\det M) \beta_{m_1}(x_{1,q}^\delta).$$
By condition (x3) for $\ell_2$, we have $\beta_{m_1}(x_{1,q}^\delta) = \delta \bar{\phi}^{\ell_2}(x_{1,q}),$
and the assertion (1) follows.

Next, we assume $i \ge 2$, and we shall prove Proposition \ref{ind} (2). The proof is essentially the same as the proof of assertion (1). 
It is sufficient to prove the assertion when $r=i$.
We write $\mb{x}=\mb{x}^{(i)}$ and $\mb{y}=\mb{y}^{(i)}$.
Let $\mb{x}' \in R_N^{h-i+1}$ be the vector obtained from $\mb{x}$ by eliminating the $m_j$-th rows for $j=1,...,i-1$, and $\mb{y}'$ the vector obtained from $\mb{y}$ by eliminating the $n_k$-th rows  for $k=1,...,i-1$. 
Since the $m_r$-th rows of $\mb{x}$ are 0 for all $r$ with $1 \le r \le i-1$ by Lemma \ref{ab} (1), we have $\mb{y}'=M_{i-1}\mb{x}'.$ 
We assume the $m_i'$-th component of $\mb{x}'$ corresponds to the $m_i$-th component of $\mb{x}$, and the $n_i'$-th component of $\mb{y}'$ corresponds to the $n_i$-th component of $\mb{y}$. 
By Lemma \ref{ab} (2) and Proposition \ref{[] and phi and x} (2), we have 
$$\mb{y}'=-\delta \bar{\phi}^{\ell_i}(x_{\nu_{i-1},q})\mb{e'}_{n_i '}^{(m)},$$
where $( \mb{e'}_i^{(m)} )_{i=1}^{h-i+1}$ denotes the standard basis of $R_N^{h-i+1}$.

Let $\widetilde{M}_{i-1}$ be the matrix of cofactors of $M_{i-1}$. Multiplying the both sides of $\mb{y}'=M_{i-1}\mb{x}'$ by $\widetilde{M}_{i-1}$,
and comparing the $m_i '$-th components, we obtain 
$$(-1)^{n_i'+m_i'+1}(\det M_i) \delta \bar{\phi}^{\ell_i}(x_{\nu_{i-1},q})=(\det M_{i-1}) \beta_{m_i}(x_{\nu,q}^\delta).$$
By condition (x3) for $\ell_{i+1}$, and since $x_{n,q}^\delta $ is an element of $\mca{F}_{Q\nu,N}$, we have $$\beta_{m_i}(x_{\nu,q}^\delta) = \delta \bar{\phi}^{\ell_{i+1}}(x_{\nu,q}).$$
Here, the proof of Proposition \ref{ind} is complete.
\end{proof}  

\subsection{} 
Now we prove Theorem \ref{Main theorem}.
\begin{proof}[Proof of Theorem \ref{Main theorem} ]

We may vary $m$ in the proof of Theorem \ref{Main theorem}. So, the element $\bar{\phi}^{\ell_{r+1}}(x_{\nu_r,q}) \in R_N=(\bb{Z}/p^{N_m})[\Gal(F_m/F_0)]_\chi$ defined in \S 6.2 is denoted by $\bar{\phi}^{\ell_{r+1}}(x_{\nu_r,q})_m$.

By induction on $r$, we shall prove that $ \big( \bar{\phi}^{\ell_{r+1}}(x_{\nu_{r},q})_m \big)_{m\ge 0}$ converges to $\pm \delta \det M_r \in \Lambda_{\chi}$, where a sequence $(a_m)_{m \ge 0} \in \prod_{m \ge 0} R_{F_m,N_m,\chi}$ is said to {\em converge} to $b=(b_m)_{m \ge 0} \in \plim_{m \ge 0}R_{F_m,N_m,\chi}=\Lambda_{\chi}$ if for each $m$, there exists an integer $L_m$ such that the image of $a_{m'}$ in $R_{F_m,N_m,\chi}$ coincides to $b_m \in R_{F_m,N_m,\chi}$ for any $m' \ge L_m$.

First, we consider the equality $\delta (\det M) \bar{\phi}^{\ell_2}(x_{1,q})_m= \pm \delta^2 (\det M_1) \bar{\varphi}_{F,N, \chi}(\mrm{cyc}(\rho_m)_\chi).$
Since the right hand side converges to $\pm \delta^2(\det M_1)(\det M)$ and both $\delta$ and $\det M$ are non-zero element, we find that $\big( \bar{\phi}^{\ell_2}(x_{1,q})_m \big)_{m\ge 0}$ converges to $\pm \delta \det M_1$.
(Note the sign $\pm$ does not depend on $m$, see Proposition \ref{ind}). 

Next, we assume that $\big( \bar{\phi}^{\ell_r}(x_{\nu_{r-1},q})_m \big)$ converges to $\pm \delta \det M_{r-1}$.
Then, the right hand side of $\delta (\det M_{r-1}) \bar{\phi}^{\ell_{r+1}}(x_{\nu_r,q}) = \pm \delta (\det M_{r}) \bar{\phi}^{\ell_{r}}(x_{\nu_{r-1},q})$
converges to \\ $\pm \delta^2(\det M_r)(\det M_{r-1})$.
Since we take $\det M_{r-1}\not= 0$, 
the sequence $\big( \bar{\phi}^{\ell_{r+1}}(x_{\nu_r,q})_m \big)_{m \ge 0}$ converges to $\pm \delta \det M_r$.

By induction, we conclude $\big( \bar{\phi}^{\ell_{i+1}}(x_{\nu,q})_m \big)$ converges to $\pm \delta \det M_i$.
Since $(x_{\nu,q})_m \in W_{F_m,N_m}^{q\nu }$ with $\epsilon (q\nu)=i$, we have $\bar{\phi}^{\ell_{i+1}}(x_{\nu,q})_m \in\mf{C}_{i,F_m,N \chi}$ for all $m \in \bb{Z}_{\ge 0}$. 
Hence we have $\delta \det M_i \in \mf{C}_{i,\chi}$. This completes the proof of Theorem \ref{Main theorem}.
\end{proof}

\begin{rem}\label{structure of trivial character part}
We remark on the higher Fitting ideals of the trivial character parts of $X$ and $X'$.
Let $1_{\Delta}$ be the trivial character in $\widehat{\Delta}$.
It is a well-known fact that $X_{1_{\Delta}}=0$ (cf.\ \cite{Wa} Proposition 15.43).
In particular, we have $$\Fitt_{\Lambda_{1_\Delta},i}(X_{1_\Delta})= \Fitt_{\Lambda_{1_\Delta},i}(X'_{1_\Delta})=\Lambda_{1_\Delta}$$
for $i \ge 0$.
\end{rem}

\subsection{}
Here, we remark on some application of Theorem \ref{Main theorem}.

For each ideal $I$ of $\Lambda_\chi$, we denote the unique minimal principal ideal of $\Lambda_\chi$ containing $I$ by $\mca{P}(I)$. Since $\Lambda_\chi$ is UFD, the ideals $\mca{P}(I)$ are well-defined.
Recall that we denote the smallest number of generators of an $R$-module $M$ by $\nu_R(M)$ (cf.\ \S 3). 
We obtain the following corollary of Theorem \ref{Main theorem}.

\begin{cor}\label{on infty}
Let $\chi $ be a non-trivial character in $\widehat \Delta$. 
Then, $$\Fitt_{\Lambda_\chi,i}(X_{\chi}) \subseteq  \Fitt_{\Lambda_\chi,i}(X'_{\chi}) \subseteq \mca{P}(\mf{C}_{i,\chi})$$
for all $i \ge 0$. In particular, if $\mca{P}(\mf{C}_{i,\chi})\neq\Lambda_\chi$, then we have $\nu_{\Lambda_\chi}(X_{\chi})\ge i+1$ and $\nu_{\Lambda_\chi}(X'_{\chi}) \ge i+1$. 
\end{cor}
\begin{proof}
Since $X'_{\chi}$ is a quotient module of $X_{\chi}$, we obtain the first inclusion   
$$\Fitt_{\Lambda_\chi,i}(X_{\chi}) \subseteq  \Fitt_{\Lambda_\chi,i}(X'_{\chi}).$$
The second inclusion follows from Theorem \ref{Main theorem} immediately.
\end{proof}

\begin{rem}\label{equivalence between Th and IMC}
Let $\chi \in \widehat\Delta$ be a non-trivial character. 
Since $X_{\chi}$ is pseudo-isomorphic to $X'_{\chi}$, we have 
$\cha_{\Lambda_\chi}(X_{\chi})=\cha_{\Lambda_\chi}(X'_{\chi}).$
Since $\ann_{\Lambda_{\chi}} (X_{\mrm{fin}, \chi})$ is an ideal of $\Lambda_\chi$ whose index is finite, we have 
$\cha_{\Lambda_\chi}(X'_{\chi})=\Fitt_{\Lambda_\chi,0}(X'_{\chi})=\mca{P}(\mf{C}_{0,\chi})$
by Lemma \ref{no pn} and Theorem \ref{Main theorem} for $i=0$. By Proposition \ref{size}, we have 
$\mca{P}(\mf{C}_{0,\chi})=\cha_{\Lambda_{\chi}}\big( (\overline{\mca{O}_{\infty}^1}/\overline{C_{\infty}^1})_{\chi}\big).$
Hence Theorem 1.1 is a refinement of the Iwasawa main conjecture.
\end{rem}

From Corollary \ref{on infty}, we obtain the following results on the higher Fitting ideals of $A_{m,\chi}$ and $A'_{m,\chi}$.

\begin{cor}\label{on m}
Let $\chi $ be a non-trivial character in $\widehat \Delta$. For each $m \ge 0$, we denote the image of $\mca{P}(\mf{C}_{i,\chi})$ in $R_{F_m,\chi}=\bb{Z}_p[\Gal(F_m/\bb{Q})]_\chi$ by $\mca{P}(\mf{C}_{i,\chi})_m$.
Then,   
$$\Fitt_{R_{F_m,\chi},i}(A_{m,\chi}) \subseteq  \Fitt_{R_{F_m,\chi},i}(A'_{m,\chi}) \subseteq \mca{P}(\mf{C}_{i,\chi})_m$$
for all $m \ge 0$ and $i \ge 0$. In particular, if $\mca{P}(\mf{C}_{i,\chi})\neq\Lambda_\chi$, then we have $\nu_{R_{F_m,\chi}}(A_{m,\chi})\ge i+1$ and $\nu_{\Lambda_\chi}(A'_{m,\chi}) \ge i+1$. 
\end{cor}

\begin{rem}
We can know more about $\nu_{R_{F_m,\chi}}(A_{m,\chi})$ and $\nu_{R_{F_m,\chi}}(A_{m,\chi})$ by direct application of the result in \cite{MR} (see Remark \ref{0-layer}) without using Theorem \ref{Main theorem}.  
Let $\chi \in \widehat \Delta$ be a non-trivial character. Suppose $N$ is sufficiently large. 
By Nakayama's lemma, the following conditions are equivalent:
\begin{enumerate}
\item ${\rm dim}_{\bb{F}_p}A_{0,\chi}/p=i+1$;
\item $\nu_{\Lambda_\chi}(X_{\chi})=i+1$;
\item $\nu_{R_{F_m,\chi}}(A_{m,\chi})=i+1$;
\item $\Fitt_{\Lambda_\chi,i}(X_{\chi})\neq\Lambda_\chi$ and $\Fitt_{\Lambda_\chi,i+1}(X_{\chi})=\Lambda_\chi$;
\item $\Fitt_{R_{F_m,\chi},i}(A_{m,\chi})\neq R_{F_m,\chi}$ and $\Fitt_{R_{F_m,\chi},i+1}(A_{m,\chi})=R_{F_m,\chi}$;
\item $\mf{C}_{i,F_m,N}\neq R_{F_m,\chi}$ and $\mf{C}_{i+1,F_m,N}=R_{F_m,\chi}$.
\end{enumerate}
\end{rem}

\end{document}